\documentclass[10pt]{article}
\usepackage{amsmath, amsthm, amssymb, amsfonts, mathrsfs, mathtools, upgreek}
\usepackage{graphicx}
\usepackage{makeidx}
\usepackage[left=2cm,right=2cm,top=2cm,bottom=2cm]{geometry}
\usepackage{caption}
\usepackage{subcaption}
\usepackage[section]{placeins}
\usepackage{enumitem}
\usepackage{tikz}
\usepackage{stmaryrd}
\usepackage{hyperref}
\usepackage{upquote}

\usepackage{hyperref}
\hypersetup{colorlinks=true, citecolor={blue}, linkcolor=blue, filecolor=magenta, urlcolor=cyan}

\newtheorem{theorem}{Theorem}[section]
\newtheorem*{thm}{Theorem}

\newtheorem{lema}[theorem]{Lemma}
\newtheorem{corollary}{Corollary}[theorem]
\theoremstyle{definition}
\newtheorem{defi}{Definition}

\numberwithin{equation}{section}
\newcommand{\ip}[2]{\langle #1, #2 \rangle}
\newcommand{\divides}{\mid}
\newcommand{\srg}{\operatorname{SRG}}
\newcommand{\cay}{\operatorname{Cay}}
\newcommand{\spec}{\operatorname{Spec}}
\newcommand{\lcm}{\operatorname{lcm}}
\newcommand{\uc}{\operatorname{UC}}
\newcommand{\iu}{{\mathbf{i}}}

\def \m {{|}}

\def \Zl {{\mathbb Z}}
\def \Nl {{\mathbb N}}

\def \Zl {{\mathbb Z}}
\def \Ql {{\mathbb Q}}
\def \Cl {{\mathbb C}}

\def \ld {{\lambda}}

\def \eu {{\textbf{e}_u}}
\def \ev {{\textbf{e}_v}}
\title{Grover walks on unitary Cayley graphs and integral regular graphs}

\author{ Koushik Bhakta and Bikash Bhattacharjya\\
	Department of Mathematics\\
	Indian Institute of Technology Guwahati, India\\
	b.koushik@iitg.ac.in, b.bikash@iitg.ac.in }

\date{}
\begin{document}
	\maketitle
	
	\vspace{-0.3in}
	
	\begin{center}{\textbf{Abstract}}\end{center}
	\noindent The unitary Cayley graph  has vertex set $\{0,1, \hdots ,n-1\}$, where two vertices  $u$ and $v$ are adjacent if $\gcd(u - v, n) = 1$. In this paper, we study periodicity and perfect state transfer of Grover walks on the unitary Cayley graphs. We characterize all periodic unitary Cayley graphs. We prove that periodicity is a necessary condition for occurrence of perfect state transfer on a vertex-transitive graph. Also, we provide a necessary and sufficient condition for the occurrence of perfect state transfer on circulant graphs. Using these, we prove that only four graphs in the class of unitary Cayley graphs exhibit perfect state transfer. Also, we provide a spectral characterization of the periodicity of Grover walks on integral regular graphs.

	\vspace*{0.3cm}
	
	\section{Introduction}
	
	The study of quantum walk \cite{quantum} on graphs is an important area that lies at the intersection of quantum computing and graph theory. It serves as a fundamental building block for various quantum algorithms and has applications in fields such as computer science, physics, and even biology. Discrete-time quantum walks \cite{random} is a quantum analogue of the classical random walk, where a particle moves through space in discrete steps according to specific rules. In a classical random walk, a particle moves from one vertex to another in a graph with a probability determined by the edges connecting that vertex. However, in a quantum walk, the movement of the particle is governed by the principles of quantum mechanics, allowing it to exhibit unique behaviors such as superposition and interference. Quantum walks have been extensively studied in both physics and mathematics. One of the most exciting properties of quantum walks is their capability for perfect state transfer, which essentially involves transferring between specific quantum states with a probability 1. Periodicity is a case of state transfer in which any initial state returns to its original state at a particular time. There are two types of quantum walks: continuous-time quantum walk and discrete-time quantum walk. The continuous-time quantum walk has been explored for a relatively long time, see \cite{circulant1, circulant2, periodic, state, when}. We also find studies of perfect state transfer in discrete-time quantum walks, though not as much as in continuous-time quantum walks, see \cite{godsildct, dqw3, dqw1}. Recently, Chan and Zhan \cite{pgstdqw} defined pretty good state transfer in discrete-time quantum walks.

	This paper explores the concepts of periodicity and perfect state transfer in the Grover walks, a type of discrete-time quantum walks. We say that a graph is a \emph{$k$-periodic Grover walk} if and only if $k$ is the smallest positive integer such that the $k$-th power of the time evolution operator becomes the identity. The research on this topic has become very active in the past decade. Higuchi et al. \cite{discrete} showed that the Grover walks on the complete graph $K_n$ is periodic if and only if $n=2$ or $n=3$. Also, they proved that the Grover walks on the complete bipartite graph $K_{r,s}$ is periodic for any $r,s\in \Nl$. Ito et al. \cite{completegraph} proved that the Grover walks on complete graphs on $n$ vertices with a self-loop at each vertex is periodic with period $2n$. In \cite{bethetrees}, Kubota et al. provided a comprehensive characterization of the periodicity of generalized Bethe trees. Yoshie \cite{distance} studied the periodicity of Grover walks on distance regular graphs. Kubota \cite{bipartite} explored the periodicity of Grover walks on regular bipartite graphs with at most five distinct adjacency eigenvalues. Also, he explored Grover walks on regular mixed graphs in \cite{kubota}. There are a few more studies on the periodicity of Grover walks, see \cite{qq, yoshie, oddperiodic}. 
	
	A graph exhibiting perfect state transfer is quite important due to its application in quantum information processing. Zhan \cite{dqw1} established an infinite family of $4$-regular circulant graphs where perfect state transfer occurs. Barr et al. \cite{barr} studied the perfect state transfer of Grover walks on variants of cycles. Kubota et al. \cite{pstdc} established a necessary condition about the eigenvalues of a graph to enable perfect state transfer between states associated with vertices. Additionally, they provided necessary and sufficient conditions for occurrence perfect state transfer in complete multipartite graphs. In this paper, we study periodicity and perfect state transfer on unitary Cayley graphs. The following are our main results. For terminologies and proof, see later sections.
	\begin{thm}[\em{Theorem} 3.6]
		The unitary Cayley graph $\uc(n)$ is periodic if and only if  $n=2^\alpha 3^\beta$, where $\alpha ~\text{and}~ \beta $ are non-negative integers with $\alpha + \beta \neq 0 $.
	\end{thm} 
	\begin{thm}[\em{Theorem} 4.2]
		Let $\mu_0,\mu_1,\hdots,\mu_{n-1}$  be the  eigenvalues of the discriminant of a circulant graph $\cay(\Zl_n, C)$, where $\mu_j=\frac{1}{|C|}\sum_{s\in C} e^{\frac{2\pi j s \iu }{n}}$ for $j\in\{0,\hdots, n-1\}$. Then prefect state transfer occurs in $\cay(\Zl_n, C)$ from a vertex $u$ to another vertex $v$ at time $\tau$ if and only if all of the following conditions hold.
	\begin{enumerate}[label=(\roman*)]
		\item $n$ is even and $u-v=\frac{n}{2}$.
		\item $T_\tau(\mu_j)=\pm1$ for $j\in\{0,\hdots,n-1\}$, where $T_n(x)$ is the Chebyshev polynomial of the first kind.
		\item $T_\tau(\mu_j)\neq T_\tau(\mu_{j+1})$ for $j\in\{0,\hdots,n-2\}.$
	\end{enumerate}
	\end{thm}
	\begin{thm}[\em{Theorem} 4.5]
		The only unitary Cayley graphs $\uc(n)$ exhibiting perfect state transfer are $K_2,~C_4,~C_6$ and $\uc(12)$. 
	\end{thm}
	
	The paper is organized as follows. In Section 2, we initially give the definition of a few matrices through which the terms \emph{periodicity} and \emph{perfect state transfer} are defined. We also express the eigenvalues of unitary Cayley graphs in terms of Ramanujan sums. In Section 3, we first provide a necessary and sufficient condition for an integral regular graph to be periodic. Using the Ramanujan sums, we classify the periodic unitary Cayley graphs. In Section 4, we prove that periodicity is a necessary condition for the occurrence of perfect state transfer in a vertex-transitive graph. Also, we determine a necessary and sufficient condition for exhibiting perfect state transfer on circulant graphs. We characterize complete graphs and cycles exhibiting perfect state transfer. Finally, we conclude that only four graphs in the class of unitary Cayley graphs exhibit perfect state transfer. In Section 5, we characterize integral regular graphs that are periodic.

	\section{Preliminaries}\label{prel}

	Let $G:=(V(G),E(G))$ be a finite, simple and connected graph with vertex set $V(G)$ and edge set $E(G)$. Note that $G$ is a graph without loops, parallel edges and have no orientations. We write the elements of $E(G)$ as $uv$, where $u,v\in V$, $u\neq v$ and, $uv$ and $vu$ represent the same edge. That is, $E(G)\subseteq \{uv: u,v\in V(G), u\neq v\}$ with the convention that $uv=vu$. Let $A:=A(G)\in \Cl^{V(G)\times V(G)}$ be the \emph{adjacency matrix} of $G$, where
	$$A_{uv} = \left\{ \begin{array}{rl}
		1 &\mbox{ if }
		uv\in E \\ 
		0 &\textnormal{ otherwise.}
	\end{array}\right.$$  
	Apart from the adjacency matrix, we define a few more matrices associated to a graph in the next sub-section.  The size of these matrices is clear from their context. The graphs $K_n$ and $C_n$ are the complete graph and the cycle on $n$ vertices. Similarly, $K_{m,m}$ and $K_{m,m,m}$ are the complete $m$-regular bipartite graph on $2m$ vertices and the complete $2m$-regular tripartite graph on $3m$ vertices, respectively.  The matrices $J$ and $I$ denote the all-one matrix and the identity matrix, respectively.

	\subsection{Grover walks}
	If $uv$ is an edge of a graph $G$, then the ordered pairs $(u,v)$ and $(v,u)$ are called the \emph{arcs} of $G$ associated to the edge $uv$. We define $\mathcal{A}=\{(u, v), (v, u):uv\in E(G) \}$, the set of all symmetric arcs of $G$. The vertices $u$ and $v$ are called the \emph{origin} and \emph{terminus} of an arc $(u,v)$, respectively. Let $a$ be an arc of $G$, where $a=(u,v)$. We write $o(a)$ and $t(a)$ to denote the origin and terminus of $a$, respectively, that is, $o(a)=u$ and $t(a)=v$.  The inverse arc of $a$, denoted $a^{-1}$, is the arc  $(v,u)$.
	
	We now define a few matrices required for the definition of Grover walks. The \emph{boundary matrix} $d:=d(G)\in \mathbb{C}^{V(G)\times \mathcal{A}}$ of $G$ is defined by $$d_{xa}=\frac{1}{\sqrt{\deg x}}\delta _{x, t(a)},$$ where $\delta_{a,b}$ is the Kronecker delta function. The \emph{shift matrix} $S:=S(G)\in \mathbb{C}^{\mathcal{A} \times \mathcal{A}}$ of $G$ is defined by $$S_{ab}=\delta_{a,b^{-1}}.$$ Define the \emph{time evolution matrix} $U:=U(G)\in \mathbb{C}^{\mathcal{A}\times \mathcal{A}}$ of $G$ by $$U=S(2d^*d-I).$$ The time evolution matrix is also known as the Grover transition matrix. A discrete-time quantum walk on a graph G is determined by a unitary matrix, which acts on the complex functions of the symmetric arcs of G. The discrete-time quantum walks defined by $U$ are called the \emph{Grover walks}. The entries of the time evolution matrix are calculated as $$U_{ab}=\frac{2}{\deg t(b)}\delta_{o(a),t(b)}-\delta_{a,b^{-1}}.$$ The \emph{discriminant} $P:=P(G)\in \mathbb{C}^{V(G)\times V(G)}$ of $G$ is defined by $$P=dSd^*.$$
	See \cite{mixedpaths} for more details about the matrices $d,~S,~U$ and $P$. If  $G$ is a regular graph, then the matrix $P$ can be expressed in terms of its adjacency matrix $A$.
	\begin{lema}\cite{qq}  \label{reg}
		If $G$ is a $k$-regular graph, then $P=\frac{1}{k}A$. Further, the modulus of eigenvalues of $P$ is less than or equal to $1$.
	\end{lema}
	\begin{defi}
		A graph $G$ is \emph{periodic} if $U^\tau =I$ for some positive integer $\tau$ and $U^j \neq I$ for every $j$ with $0< j < \tau $. In such a case, the number $\tau $ is the \emph{period} of the graph, and the graph is \emph{$\tau$-periodic}.
	\end{defi}	
	We define $\spec_A(G),~ \spec_P(G)$ and $\spec_U(G)$ to be the set of all distinct eigenvalues of $A,~P$ and $U$ of the graph $G$, respectively.  The set $\spec_A(G)$ is also known as the \emph{adjacency spectrum} of $G$. A graph $G$ is said to be \emph{integral} if all the eigenvalues of its adjacency matrix are integers.

	Since $U$ is a unitary matrix, it is diagonalizable. Therefore the periodicity of a graph can be easily determined by the eigenvalues of its time evolution matrix.
	\begin{lema}\cite{mixedpaths} \label{period}
		A graph $G$ is a $\tau$-periodic graph if and only if $ \lambda^\tau  =1$ for every $\lambda \in \spec_U(G)$, and for each $j\in\{1,\hdots,\tau-1\}$, $\lambda^j\neq 1$ for for some $\ld\in\spec_U(G)$.
	\end{lema}
	From the previous lemma, we observe that the period of G can be determined from the eigenvalues of $U$.
	\begin{corollary} \label{periodic}
		Let $\ld_1, \hdots , \ld_t$ be the distinct eigenvalues of the time evolution matrix of a periodic graph $G$. Let $k_1, \hdots, k_t$ be the least positive integers such that $\ld_1 ^{k_1}=1, \hdots,\ld_t^{k_t}=1$. Then the period of $G$ is $\lcm(k_1, \hdots, k_t)$.
	\end{corollary}
	The following result is known as the spectral mapping theorem of the Grover walks. For a non-negative integer $k$, by $\{a\}^k$, we denote the multi-set $\{a,\ldots,a\}$, where the element $a$ repeats $k$-times. We also denote $\sqrt{-1}$ by $\iu$.
	\begin{lema} \label{evu}  \cite{higu}
		Let $\mu_1, \hdots,\mu_n$ be the eigenvalues of the discriminant of a graph $G$. Then the multi-set of eigenvalues of the time evolution matrix is $$\left\{e^{\pm \iu \arccos(\mu_j)}:j\in\{1,\hdots,n\}\right\}\cup \{1\}^{b_1} \cup \{ -1\}^{b_1-1+1_B},$$ where $b_1=|E(G)|-|V(G)|+1$, and $1_B=1$ or $0$ according as $G$ is bipartite or not.	
	\end{lema}
	

	For $\Phi,\Uppsi\in \Cl^{\mathcal{A}}$, we denote by $\ip{\Phi}{\Uppsi}$ the Euclidean inner product of $\Phi$ and $\Uppsi$. Let $G$ be a graph and $U$ be its time evolution matrix. A vector $\Phi\in \Cl^\mathcal{A}$ is said to be a \emph{state} if $\ip{\Phi}{\Phi}=1$. We say that perfect state transfer occurs from a state $\Phi$ to another state $\Psi$ at time $\tau\in \Nl$ if there exists a unimodular complex number $\gamma$ such that $$U^\tau\Phi=\gamma \Psi.$$ 
	\begin{lema} \cite{pgstdqw}\label{st11}
		Let $G$ be a graph and $U$ be its time evolution matrix. Then perfect state transfer occurs from a state $\Phi$ to another state $\Psi$ at time $\tau$ if and only if  $|\ip{U^\tau \Phi}{\Uppsi}|=1.$ 
	\end{lema}
	Let $G$ be a graph with vertex set $V(G)$. A state $\Phi\in \Cl^\mathcal{A}$ is said to be \emph{vertex type} of $G$ if there exists a vertex $u\in V(G)$ such that $\Phi=d^* \eu$, where $\eu$ is the unit vector defined by $(\eu)_x=\delta_{u,x}$. We denote by $\chi$ the set of all vertex type states, that is, $\chi =\{d^*\eu: u\in V(G)\}$. In this paper, we deal only with the vertex type states.
	\begin{defi}
		A graph exhibits \emph{perfect state transfer} if perfect state transfer occurs between two distinct vertex type states of the graph.
	\end{defi}
	From now onward, we say that  perfect state transfer occurs from the vertex $u$ to another vertex $v$ of a graph to mean that perfect state transfer occurs from the state $d^*\eu$ to the state $d^*\ev$ of the graph.
	In Grover walks, the occurrence of perfect state transfer between vertex type states has strong connection with Chebyshev polynomials.

	The \emph{Chebyshev polynomial of the first kind}, denoted $T_n(x)$, is the polynomial defined by $T_0(x)=1$, $T_1(x)=x$ and $$T_n(x)=2xT_{n-1}(x)-T_{n-2}(x) ~\text{for}~ n \geq 2.$$ It is well known that 
	\begin{equation}\label{chb}
		T_n(\cos\theta)=\cos(n\theta).
	\end{equation}
	This implies that $|T_n(x)|\leq 1$ for $|x|\leq 1$. Thus, the next lemma follows easily.
	\begin{lema}\label{ch}
		Let $\mu\in[-1,1]$ and $\tau$ be any positive integer. Then
		\begin{enumerate}[label=(\roman*)]
			\item $|T_\tau (\mu)|\leq 1$.
			\item $T_\tau (\mu)= 1$ if and only if $\mu =\cos\frac{s}{\tau}\pi$ for some even positive integer $s$.
			\item $T_\tau (\mu)= -1$ if and only if $\mu =\cos\frac{s}{\tau}\pi$ for some odd positive integer $s$.
		\end{enumerate}
	\end{lema}

	In \cite{pstdc}, Kubota and Segawa studied perfect state transfer between vertex type states via Chebyshev polynomial of the first kind. They gave a necessary condition on the eigenvalues of a graph for perfect state transfer between vertex-type states to occur.
	\begin{lema} \cite{pstdc}\label{ch11}
		Let $T_n(x)$ be the Chebyshev polynomial of the first kind.	Let $G$ be a graph with the time evolution matrix $U$ and discriminant $P$. Then $dU^\tau d^* = T_\tau (P)$ for $\tau \in \Nl \cup \{0\}$. 
	\end{lema}
	Since the discriminant $P$ of a graph is a symmetric matrix, it has a spectral decomposition. Suppose $P$ has distinct eigenvalues $\mu_1, \hdots, \mu_m$ and let $E_r$ be the matrix that represents the orthogonal projection of $P$ on the eigenspace associated with $\mu_r$ for $1\leq r\leq m$. Then the spectral decomposition of $P$ is
	\begin{equation*}
		P=\sum_{r=1}^{m} \mu_r E_r. 
	\end{equation*}
	The following properties are satisfied by these eigenprojectors.
	\begin{enumerate}[label=(\roman*)]
		\item $E_r^2=E_r$ for $1\leq r \leq m$,
		\item $E_rE_s=0$ for $r\neq s$, $1\leq r,~s \leq m$,
		\item $E_r^t=E_r$ for $1\leq r \leq m$, and
		\item $E_1+\hdots + E_m=I$.
	\end{enumerate}
	Using the preceding properties of the eigenprojectors, if $f(x)$ is a polynomial, then we have \begin{equation}\label{sd}
		f(P)=\sum_{r=1}^{m} f(\mu_r) E_r.
	\end{equation} 
	For a vertex $u$ of a graph $G$, we define $\Theta_P(u)=\{\mu_r\in \spec_P(G): E_r \eu \neq 0\}.$ The set $\Theta_P(u)$ is called the \emph{eigenvalue support} of the vertex $u$ with respect to $P$.
	\begin{theorem}\cite{pstdc}\label{p1}
		Let  $u$ and $v$ be two distinct vertices of a graph $G$ and $P$ be its discriminant. If perfect state transfer occurs from $u$ to $v$ at time $\tau$, then $T_\tau(\mu)=\pm 1$ for all $\mu \in \Theta_P(u)$.
	\end{theorem}
	
	
	\subsection{Unitary Cayley graphs}	
	
	Let $(\Gamma,+)$ be a finite abelian group and $C$ be an inverse closed subset of $\Gamma\setminus\{0\}$, where $0$ is the identity element of $\Gamma$. The \emph{Cayley graph} $\cay(\Gamma, C)$ of $\Gamma$ with respect to $C$ is an undirected graph whose vertex set is $\Gamma$ and two vertices $u$ and $v$ are adjacent if and only if $u-v\in C$. If $\Gamma=\Zl_{n}$, then the Cayley graph is called a \emph{circulant graph}. The Cayley graph $\cay(\Gamma, C)$ is a regular graph of degree $\m C\m $. The graph $\cay(\Gamma, C)$ is connected if $C$ generates $\Gamma$. Now, we are going to describe unitary Cayley graphs.

	Let $\Zl_n$ be the additive group of integers modulo $n$, where $n\geq 2$. We write $\Zl_n=\{0, 1,  ..., n-1\}$ and $U_n=\{a\in\Zl_n:\gcd(a,n)=1\}$. Then the Cayley graph $\cay(\Zl_n, U_n)$ is called the \emph{unitary Cayley graph} on $n$ vertices. We prefer to denote $\cay(\Zl_n,U_n)$ by $\uc(n)$. The graph $\uc(n)$ is a connected and $\varphi(n)$-regular graph, where $\varphi$ denotes the Euler's totient function. Note that the unitary Cayley graph is a circulant graph. It is well known that $\uc(n)$ is an integral graph. For more information about unitary Cayley graphs, see \cite{ucg}.
	
	The eigenvalues and corresponding eigenvectors (see \cite{rep}) of the adjacency matrix of $\cay(\Zl_n,C)$ are given by  \begin{equation}\label{evc}
		\lambda_j=\sum_{s\in C} \omega_n^{js}~~ \text{and}  ~~ \mathbf{v}_j=\left[1~~ \omega_n^j~~ \omega_n^{2j}~~ ...~~ \omega_n^{(n-1)j}\right]^t~~\text{for}~0\leq j\leq n-1,
	\end{equation} where $\omega_n=e^{\frac{2\pi \iu}{n}}$. Note that $\{\omega_n^r:0\leq r\leq n-1\}$ is the set of all (complex) solutions of the equation $x^n=1$.
	
	In 1918, Ramanujan introduced a sum (now known as the Ramanujan sum) in his published seminar paper \cite{ramanujan}, defined by $$ R(j,n)=\sum_{r\in U_n} \omega_n^{jr}=\mathop{\sum_{r\in U_n }}_{r<\frac{n}{2}} 2 \cos\left( \frac{2\pi jr }{n} \right) ~\text{for}~n,j\in \Zl,~ n\geq 1.$$
	\begin{theorem} \cite{ucg} \label{main}
		Let $\ld_0, \ld_1, \hdots, \ld_{n-1}$ be the eigenvalues of the adjacency matrix of the unitary Cayley graph  $\uc(n)$. Then $\lambda_j=R(j,n)~~\text{for}~~ 0\leq j\leq n-1.$
	\end{theorem}
	The Ramanujan sum $R(j,n)$ can also be expressed in terms of arithmetic functions.
	\begin{theorem}\cite{ramanujan} \label{ramanujan}
		For integers $n$, $j$ with $n\geq 1$, $$ R(j,n)=\sum_{r\divides \gcd(j,n)} \mu\left(\frac{n}{r}\right)= \mu (c_{n,j})\frac{\varphi(n)}{\varphi(c_{n,j})}, ~~\text{ where}~~ c_{n,j}=\frac{n}{\gcd(n,j)}$$ and $\mu $ is the M\"{o}bius function.
	\end{theorem}


	\section{Periodicity on unitary Cayley graphs}\label{ucg}
	A complex number $\zeta$ is said to be an \emph{algebraic integer} if there exists an integer monic polynomial $q(x)$ such that $q(\zeta)=0$. Let $\Delta$ denote the set of all algebraic integers. Note that $\spec_A(G)\subset \Delta$ for the adjacency matrix of any graph $G$, as the characteristic polynomial of $G$ is monic with integer coefficients. It is well known that $\Delta \cap \Ql = \Zl $, and $\Delta$ is a subring of $\Cl$. The following result appears implicitly in \cite{distance}. However, our proof is quite different from that in \cite{distance}.

	\begin{theorem}\label{thm1}
		A $k$-regular integral graph $G$ is periodic if and only if $\spec_A(G) \subseteq \{ \pm k, \pm \frac{k}{2}, 0\}$.
	\end{theorem}
	\begin{proof}
		Let $G$ be a $k$-regular integral periodic graph, and let $\mu\in \spec_U(G)$. Since $G$ is periodic, $\mu ^ r = 1$ for some positive integer $r$. Therefore $\mu$ is a root of $x^r-1$, that is, $\mu \in \Delta$. Thus $\spec_U(G)\subset\Delta$. 	
		
		Let $\lambda \in \spec_A(G)$. Then by Lemma \ref{reg}, $\frac{\lambda}{k}$ is an eigenvalue of $P$, and  so $e^{\pm \iu \arccos(\frac{\lambda}{k})} \in \spec_U(G)\subset\Delta$. Since $\Delta$ is a ring, $$2\frac{\lambda}{k}=e^{ \iu \arccos(\frac{\lambda}{k})}+e^{- \iu \arccos(\frac{\lambda}{k})} \in \Delta .$$ Also note that $G$ is integral, and so $\ld\in\Zl$. Therefore $2\frac{\lambda}{k}\in \Ql$, and hence $2\frac{\lambda}{k}\in \Zl$. By Lemma \ref{reg}, $-2\leq 2\frac{\lambda}{k}\leq 2$. Since $2\frac{\lambda}{k}$ is an integer, we find that $2\frac{\lambda}{k} \in \{\pm 2, \pm 1, 0\}$. Thus $\lambda \in \{ \pm k, \pm \frac{k}{2}, 0\}$.
		
		Conversely, assume  $\spec_A(G) \subseteq \left\{ \pm k, \pm \frac{k}{2}, 0\right\}$. Therefore from Lemma \ref{evu},  $$\spec_U(G) \subseteq \{\pm \iu,~ e^{\pm \iu \arccos(\pm 1)}, ~ e^{\pm \iu \arccos(\pm 1/2)}\}.$$ Thus one can find a positive integer $m$ such that $\mu^m=1$ for any $\mu\in \spec_U(G)$. Therefore, $G$ is a periodic graph.
	\end{proof}
	Recall that $\uc(n)$ is $\varphi(n)$-regular. Therefore the next corollary follows easily from Theorem \ref{thm1}.
	\begin{corollary} \label{thmk}
		The	unitary Cayley graph $\uc(n)$ is periodic if and only if $$\spec_A(\uc(n)) \subseteq \left\{ \pm \varphi(n), \pm \frac{\varphi(n)}{2}, 0\right\}.$$
	\end{corollary}
	
	\begin{lema}\label{prime}
		If the unitary Cayley graph $\uc(n)$ is periodic, then the only prime factors of $n$ are $2$ or $3$.
	\end{lema}
	\begin{proof}
		Let $p$ be a prime such that  $n=p m$ for some positive integer $m$. Note that $1\leq m \leq n-1$. We consider the eigenvalue $\lambda_{m}$ of the adjacency matrix of $\uc(n)$. By Theorem \ref{ramanujan}, we have $$\lambda_{m}=R(m,n)=\mu(p)\frac{\varphi(n)}{\varphi(p)}=-\frac{\varphi(n)}{\varphi(p)},~\text{as}~ c_{n,m}=p~\text{and}~\mu(p)=-1.$$ By Corollary \ref{thmk}, it is easy to conclude that $\varphi(p)=1$ or $2$, as $\uc(n)$ is periodic. Therefore, the only possible values for the prime $p$ are $2$ and $3$.
	\end{proof}

	\begin{lema}
		Let $n=2^\alpha$ for some positive integer $\alpha$. Then the unitary Cayley graph $\uc(n)$ is periodic.
	\end{lema}
	\begin{proof}
		The eigenvalues of the adjacency matrix of $\uc(2)$ are $1$ and $-1$. Thus by Corollary \ref{thmk}, $\uc(2)$ is periodic.

		Now let $\alpha\geq2$. For $j=0$, $$\lambda_0=R(0,n)=\mu(1)\frac{\varphi(n)}{\varphi(1)}=\varphi(n).$$ For $j=2^{\alpha -1},$ we have $$\lambda_j=R(j,n)=\mu(2)\frac{\varphi(n)}{\varphi(2)}=-\varphi(n).$$  Now let $j\notin\{ 0,\ 2^{\alpha-1}\}$ and $1\leq j\leq n-1$. If $\gcd(j,n)=1$ then $c_{n,j}=n$, and so $$\lambda_j=R(j,n)=\mu(n)\frac{\varphi(n)}{\varphi(n)}=0.$$ If $\gcd(j,n)=2^r$ for some $r$ with $1\leq r\leq \alpha -2$, then $c_{n,j}=2^{\alpha-r}$. As $ \alpha - t\geq 2$, we have $\mu (c_{n,j})=0$. Hence $\lambda_j=0$. Thus we find that $$\spec_A(\uc(n))=\left\{\pm \varphi(n),\ 0\right\}.$$ Using Corollary \ref{thmk}, we conclude that $\uc(n)$ is periodic for $n=2^\alpha$ as well, where $\alpha\geq 2$.	
	\end{proof}

	\begin{lema}
		Let $n=3^\beta$ for some positive integer $\beta$. Then the unitary Cayley graph $\uc(n)$ is periodic.
	\end{lema}	
	\begin{proof}
	The eigenvalues of the adjacency matrix of $\uc(3)$ are $2$ and $-1$. Thus by Corollary \ref{thmk}, $\uc(3)$ is periodic.
		
		Now let $\beta\geq 2$. For $j=0$, $$\lambda_0=R(0,n)=\mu(1)\frac{\varphi(n)}{\varphi(1)}=\varphi(n).$$ For $j=3^{\beta -1}$, we have 
  $$\lambda_j=R(j,n)=\mu(3)\frac{\varphi(n)}{\varphi(3)}=-\frac{\varphi(n)}{2}.$$
  Now let $j\notin\{ 0,\ 3^{\beta-1}\}$ and $1\leq j\leq n-1$. If $\gcd(j,n)=1$ then $c_{n,j}=n$, and so 
  $$\lambda_j=R(j,n)=\mu(n)\frac{\varphi(n)}{\varphi(n)}=0.$$
  If $\gcd(j,n)=3^r$ for some $r$ with $1\leq r\leq \beta -1 $, then $c_{n,j}=3^{\beta-r}$. If $\beta - r\geq 2$ then $\mu (c_{n,j})=0$, and so $\lambda_j=0$. If $\beta - r= 1$, then the only possible value of $j$ is $2\times3^{\beta -1} $. In that case, $\ld_j = -\frac{\varphi(n)}{2}$. Therefore $$\spec_A(\uc(n))=\left\{\varphi(n),\ -\frac{\varphi(n)}{2},\ 0\right\}.$$ Hence by Corollary \ref{thmk}, $\uc(n)$ is also periodic for $n=3^\beta$, where $\beta\geq2$.	
	\end{proof}
\begin{lema}\label{lastuclemma}
		Let $n=2^\alpha  3^\beta$ for some positive integers $\alpha$ and  $\beta$. Then the unitary Cayley graph $\uc(n)$ is periodic.
	\end{lema}
	\begin{proof}
		Let $\alpha =1$ and $\beta =1$. In this case, $n=6$ and $\uc(6)$ has the adjacency spectrum $\{\pm 2, \pm 1\}$. Therefore $\uc(6)$ is periodic.

		Now let $n=2^\alpha  3^\beta$,  where $\alpha \geq 1 ~\text{and} ~ \beta \geq 1$ such that $\alpha \beta \neq 1$. For $j=0$, $$\lambda_0=R(0,n)=\varphi(n).$$ For $j=2^{\alpha} 3^{\beta -1}$, $$\lambda_j=R(j,n)=-\frac{\varphi(n)}{2}.$$ For $j=2^{\alpha-1} 3^{\beta}$, $$\lambda_j=R(j,n)=-\varphi(n).$$ For $j=2^{\alpha -1}3^{\beta -1}$, $$\lambda_j=R(j,n)=\frac{\varphi(n)}{2}.$$
		Let $j\notin \{ 0,\ 2^{\alpha} 3^{\beta -1},\ 2^{\alpha-1} 3^{\beta},\ 2^{\alpha -1}3^{\beta -1}\}$ and $ 0\leq j \leq n-1$. If $\gcd(j,n)=1$ then $c_{n,j}=n$, and so $$\lambda_j=R(j,n)=\mu(n)\frac{\varphi(n)}{\varphi(n)}=0.$$ Now let $\gcd(j,n)=2^s3^r$ for some non negative integers $s$ and $r$ such that  $0\leq s\leq \alpha$, $0\leq r\leq \beta $ and $s+r\neq 0$. Observe that either $s\leq \alpha -2$ or $r\leq \beta -2$, as $j\notin  \{0,\ 2^{\alpha} 3^{\beta -1},\ 2^{\alpha-1} 3^{\beta},\ 2^{\alpha -1}3^{\beta -1}\}$.  Therefore $c_{n,j}=2^{\alpha - s} 3^{\beta -r}$, where either $\alpha - s\geq 2~\text{or}~\beta -r \geq 2$. This gives $\mu (c_{n,j})=0$, and so $\lambda_j=0$. Thus we have 
  $$\spec_A(\uc(n))=\left\{\pm \varphi(n),\ \pm \frac{\varphi(n)}{2},\ 0\right\}.$$
  Using Corollary \ref{thmk}, we find that $\uc(n)$ is periodic for $n=2^\alpha3^\beta$ as well, where $\alpha\geq 1$, $\beta\geq1$ and $\alpha\beta\neq 1$.	
	\end{proof}
	Combining Lemma \ref{prime} to Lemma \ref{lastuclemma}, we have the next theorem.
	
	\begin{theorem}\label{main1}
		The unitary Cayley graph $\uc(n)$ is periodic if and only if  $n=2^\alpha 3^\beta$, where $\alpha ~\text{and}~ \beta $ are non-negative integers with $\alpha + \beta \neq 0 $.
	\end{theorem}	
	Note that the explicit form of the eigenvalues of the adjacency matrix of $\uc(n)$ are known for \linebreak[4] $n=2^\alpha3^\beta$. Therefore by Corollary \ref{periodic}, we easily find the period of $\uc(n)$ for $n=2^\alpha3^\beta$. For $n=2^\alpha$, we have $\spec_A(\uc(n))=\{\pm \varphi(n), 0\}$. As $\uc(n)$ is $\varphi(n)$-regular, Lemma \ref{reg} gives that $\spec_P(\uc(n))=\{\pm1,0\}.$ Therefore by Lemma \ref{evu}, $\spec_U(\uc(n))=\{\pm \iu, \pm 1\}.$ Hence the period of $\uc(n)$ is $4$ for $n=2^\alpha$. Similarly, for the other cases, the period of $\uc(n)$ is $12$.  
	
	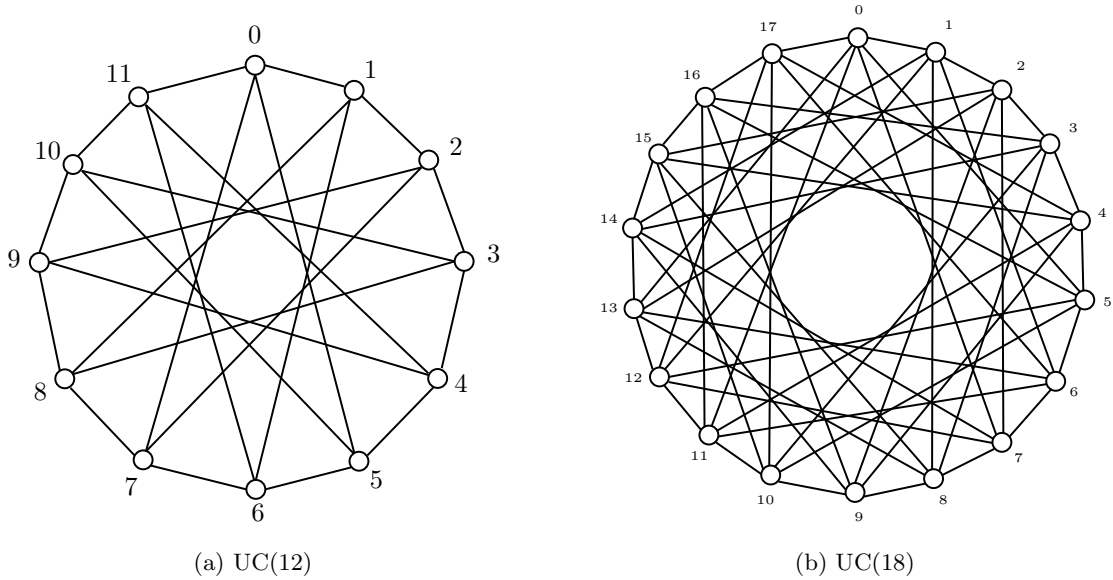
\begin{figure}[h!]
		\centering
		\begin{subfigure}{.5\textwidth}
			\centering

			\tikzset{every picture/.style={line width=0.75pt}} 
			
			\begin{tikzpicture}[x=0.65pt,y=0.65pt,yscale=-1,xscale=1]

				\draw   (193.93,88.43) .. controls (193.93,85.4) and (196.4,82.93) .. (199.43,82.93) .. controls (202.47,82.93) and (204.93,85.4) .. (204.93,88.43) .. controls (204.93,91.47) and (202.47,93.93) .. (199.43,93.93) .. controls (196.4,93.93) and (193.93,91.47) .. (193.93,88.43) -- cycle ;
				\draw   (232.13,49.03) .. controls (232.13,46) and (234.6,43.53) .. (237.63,43.53) .. controls (240.67,43.53) and (243.13,46) .. (243.13,49.03) .. controls (243.13,52.07) and (240.67,54.53) .. (237.63,54.53) .. controls (234.6,54.53) and (232.13,52.07) .. (232.13,49.03) -- cycle ;
				\draw   (299.9,30.6) .. controls (299.9,27.56) and (302.36,25.1) .. (305.4,25.1) .. controls (308.44,25.1) and (310.9,27.56) .. (310.9,30.6) .. controls (310.9,33.64) and (308.44,36.1) .. (305.4,36.1) .. controls (302.36,36.1) and (299.9,33.64) .. (299.9,30.6) -- cycle ;
				\draw   (406.13,213.03) .. controls (406.13,210) and (408.6,207.53) .. (411.63,207.53) .. controls (414.67,207.53) and (417.13,210) .. (417.13,213.03) .. controls (417.13,216.07) and (414.67,218.53) .. (411.63,218.53) .. controls (408.6,218.53) and (406.13,216.07) .. (406.13,213.03) -- cycle ;
				\draw   (174.33,145.43) .. controls (174.33,142.4) and (176.8,139.93) .. (179.83,139.93) .. controls (182.87,139.93) and (185.33,142.4) .. (185.33,145.43) .. controls (185.33,148.47) and (182.87,150.93) .. (179.83,150.93) .. controls (176.8,150.93) and (174.33,148.47) .. (174.33,145.43) -- cycle ;
				\draw   (189.13,213.23) .. controls (189.13,210.2) and (191.6,207.73) .. (194.63,207.73) .. controls (197.67,207.73) and (200.13,210.2) .. (200.13,213.23) .. controls (200.13,216.27) and (197.67,218.73) .. (194.63,218.73) .. controls (191.6,218.73) and (189.13,216.27) .. (189.13,213.23) -- cycle ;
				\draw   (357.53,45.23) .. controls (357.53,42.2) and (360,39.73) .. (363.03,39.73) .. controls (366.07,39.73) and (368.53,42.2) .. (368.53,45.23) .. controls (368.53,48.27) and (366.07,50.73) .. (363.03,50.73) .. controls (360,50.73) and (357.53,48.27) .. (357.53,45.23) -- cycle ;
				\draw   (400.93,85.83) .. controls (400.93,82.8) and (403.4,80.33) .. (406.43,80.33) .. controls (409.47,80.33) and (411.93,82.8) .. (411.93,85.83) .. controls (411.93,88.87) and (409.47,91.33) .. (406.43,91.33) .. controls (403.4,91.33) and (400.93,88.87) .. (400.93,85.83) -- cycle ;
				\draw   (421.53,144.83) .. controls (421.53,141.8) and (424,139.33) .. (427.03,139.33) .. controls (430.07,139.33) and (432.53,141.8) .. (432.53,144.83) .. controls (432.53,147.87) and (430.07,150.33) .. (427.03,150.33) .. controls (424,150.33) and (421.53,147.87) .. (421.53,144.83) -- cycle ;
				\draw   (300.33,277.63) .. controls (300.33,274.6) and (302.8,272.13) .. (305.83,272.13) .. controls (308.87,272.13) and (311.33,274.6) .. (311.33,277.63) .. controls (311.33,280.67) and (308.87,283.13) .. (305.83,283.13) .. controls (302.8,283.13) and (300.33,280.67) .. (300.33,277.63) -- cycle ;
				\draw   (360.33,261.23) .. controls (360.33,258.2) and (362.8,255.73) .. (365.83,255.73) .. controls (368.87,255.73) and (371.33,258.2) .. (371.33,261.23) .. controls (371.33,264.27) and (368.87,266.73) .. (365.83,266.73) .. controls (362.8,266.73) and (360.33,264.27) .. (360.33,261.23) -- cycle ;
				\draw   (234.73,260.43) .. controls (234.73,257.4) and (237.2,254.93) .. (240.23,254.93) .. controls (243.27,254.93) and (245.73,257.4) .. (245.73,260.43) .. controls (245.73,263.47) and (243.27,265.93) .. (240.23,265.93) .. controls (237.2,265.93) and (234.73,263.47) .. (234.73,260.43) -- cycle ;
				\draw    (243,46.4) -- (299.9,31.6) ;
				\draw    (202.4,83.4) -- (232.6,52.4) ;
				\draw    (311.33,277.63) -- (361.23,264.83) ;
				\draw    (310.9,30.6) -- (357.4,43) ;
				\draw    (402.6,82.4) -- (367.5,48.6) ;
				\draw    (179.83,139.93) -- (196.3,93) ;
				\draw    (179.83,150.93) -- (191.8,208.6) ;
				\draw    (414,207) -- (427.03,150.33) ;
				\draw    (198,218.6) -- (235.2,257.8) ;
				\draw    (244.8,263.8) -- (300.33,277.63) ;
				\draw    (427.03,139.33) -- (409.5,91) ;
				\draw    (370.8,257.8) -- (409.2,218.2) ;
				\draw    (305.4,36.1) -- (363,256) ;
				\draw    (305.4,36.1) -- (241.8,255.6) ;
				\draw    (360.6,50.9) -- (305.83,272.13) ;
				\draw    (360.6,50.9) -- (199.2,209.6) ;
				\draw    (402.2,90.1) -- (241.8,255.6) ;
				\draw    (402.2,90.1) -- (185.33,145.43) ;
				\draw    (421.53,144.83) -- (199.2,210.6) ;
				\draw    (204.6,92.1) -- (421.53,144.83) ;
				\draw    (185.33,145.43) -- (406.2,211) ;
				\draw    (241.8,53.7) -- (406.2,210) ;
				\draw    (204.6,92.1) -- (355.84,249.59) -- (363,257) ;
				\draw    (241.8,53.7) -- (304.83,272.13) ;
				
				\draw (299.6,6.2) node [anchor=north west][inner sep=0.75pt]   [align=left] {0};
				\draw (367.2,26.6) node [anchor=north west][inner sep=0.75pt]   [align=left] {1};
				\draw (416.8,71.4) node [anchor=north west][inner sep=0.75pt]   [align=left] {2};
				\draw (419.6,209.4) node [anchor=north west][inner sep=0.75pt]   [align=left] {4};
				\draw (438.8,133.8) node [anchor=north west][inner sep=0.75pt]   [align=left] {3};
				\draw (370.4,265.4) node [anchor=north west][inner sep=0.75pt]   [align=left] {5};
				\draw (301.6,284.6) node [anchor=north west][inner sep=0.75pt]   [align=left] {6};
				\draw (227.6,269) node [anchor=north west][inner sep=0.75pt]   [align=left] {7};
				\draw (174.8,213.8) node [anchor=north west][inner sep=0.75pt]   [align=left] {8};
				\draw (159.6,136.6) node [anchor=north west][inner sep=0.75pt]   [align=left] {9};
				\draw (175.2,73.4) node [anchor=north west][inner sep=0.75pt]   [align=left] {10};
				\draw (216.8,28.6) node [anchor=north west][inner sep=0.75pt]   [align=left] {11};

			\end{tikzpicture}
			\caption{$\uc(12)$}\label{fig:sub1}
		\end{subfigure}%
		\begin{subfigure}{.5\textwidth}
			\centering

			\tikzset{every picture/.style={line width=0.75pt}} 
			
			\begin{tikzpicture}[x=0.55pt,y=0.55pt,yscale=-1,xscale=1]

				\draw   (264,47) .. controls (264,43.41) and (266.91,40.5) .. (270.5,40.5) .. controls (274.09,40.5) and (277,43.41) .. (277,47) .. controls (277,50.59) and (274.09,53.5) .. (270.5,53.5) .. controls (266.91,53.5) and (264,50.59) .. (264,47) -- cycle ;
				\draw   (420,227.5) .. controls (420,223.91) and (422.91,221) .. (426.5,221) .. controls (430.09,221) and (433,223.91) .. (433,227.5) .. controls (433,231.09) and (430.09,234) .. (426.5,234) .. controls (422.91,234) and (420,231.09) .. (420,227.5) -- cycle ;
				\draw   (262,360) .. controls (262,356.41) and (264.91,353.5) .. (268.5,353.5) .. controls (272.09,353.5) and (275,356.41) .. (275,360) .. controls (275,363.59) and (272.09,366.5) .. (268.5,366.5) .. controls (264.91,366.5) and (262,363.59) .. (262,360) -- cycle ;
				\draw   (109,178) .. controls (109,174.41) and (111.91,171.5) .. (115.5,171.5) .. controls (119.09,171.5) and (122,174.41) .. (122,178) .. controls (122,181.59) and (119.09,184.5) .. (115.5,184.5) .. controls (111.91,184.5) and (109,181.59) .. (109,178) -- cycle ;
				\draw   (363,83) .. controls (363,79.41) and (365.91,76.5) .. (369.5,76.5) .. controls (373.09,76.5) and (376,79.41) .. (376,83) .. controls (376,86.59) and (373.09,89.5) .. (369.5,89.5) .. controls (365.91,89.5) and (363,86.59) .. (363,83) -- cycle ;
				\draw   (396,120) .. controls (396,116.41) and (398.91,113.5) .. (402.5,113.5) .. controls (406.09,113.5) and (409,116.41) .. (409,120) .. controls (409,123.59) and (406.09,126.5) .. (402.5,126.5) .. controls (398.91,126.5) and (396,123.59) .. (396,120) -- cycle ;
				\draw   (417,173) .. controls (417,169.41) and (419.91,166.5) .. (423.5,166.5) .. controls (427.09,166.5) and (430,169.41) .. (430,173) .. controls (430,176.59) and (427.09,179.5) .. (423.5,179.5) .. controls (419.91,179.5) and (417,176.59) .. (417,173) -- cycle ;
				\draw   (159,88) .. controls (159,84.41) and (161.91,81.5) .. (165.5,81.5) .. controls (169.09,81.5) and (172,84.41) .. (172,88) .. controls (172,91.59) and (169.09,94.5) .. (165.5,94.5) .. controls (161.91,94.5) and (159,91.59) .. (159,88) -- cycle ;
				\draw   (127,127) .. controls (127,123.41) and (129.91,120.5) .. (133.5,120.5) .. controls (137.09,120.5) and (140,123.41) .. (140,127) .. controls (140,130.59) and (137.09,133.5) .. (133.5,133.5) .. controls (129.91,133.5) and (127,130.59) .. (127,127) -- cycle ;
				\draw   (205,58) .. controls (205,54.41) and (207.91,51.5) .. (211.5,51.5) .. controls (215.09,51.5) and (218,54.41) .. (218,58) .. controls (218,61.59) and (215.09,64.5) .. (211.5,64.5) .. controls (207.91,64.5) and (205,61.59) .. (205,58) -- cycle ;
				\draw   (316,350.5) .. controls (316,346.91) and (318.91,344) .. (322.5,344) .. controls (326.09,344) and (329,346.91) .. (329,350.5) .. controls (329,354.09) and (326.09,357) .. (322.5,357) .. controls (318.91,357) and (316,354.09) .. (316,350.5) -- cycle ;
				\draw   (363,325.5) .. controls (363,321.91) and (365.91,319) .. (369.5,319) .. controls (373.09,319) and (376,321.91) .. (376,325.5) .. controls (376,329.09) and (373.09,332) .. (369.5,332) .. controls (365.91,332) and (363,329.09) .. (363,325.5) -- cycle ;
				\draw   (400,283.5) .. controls (400,279.91) and (402.91,277) .. (406.5,277) .. controls (410.09,277) and (413,279.91) .. (413,283.5) .. controls (413,287.09) and (410.09,290) .. (406.5,290) .. controls (402.91,290) and (400,287.09) .. (400,283.5) -- cycle ;
				\draw   (204,348) .. controls (204,344.41) and (206.91,341.5) .. (210.5,341.5) .. controls (214.09,341.5) and (217,344.41) .. (217,348) .. controls (217,351.59) and (214.09,354.5) .. (210.5,354.5) .. controls (206.91,354.5) and (204,351.59) .. (204,348) -- cycle ;
				\draw   (161.5,320.5) .. controls (161.5,316.91) and (164.41,314) .. (168,314) .. controls (171.59,314) and (174.5,316.91) .. (174.5,320.5) .. controls (174.5,324.09) and (171.59,327) .. (168,327) .. controls (164.41,327) and (161.5,324.09) .. (161.5,320.5) -- cycle ;
				\draw   (127.5,280.5) .. controls (127.5,276.91) and (130.41,274) .. (134,274) .. controls (137.59,274) and (140.5,276.91) .. (140.5,280.5) .. controls (140.5,284.09) and (137.59,287) .. (134,287) .. controls (130.41,287) and (127.5,284.09) .. (127.5,280.5) -- cycle ;
				\draw   (110,233.5) .. controls (110,229.91) and (112.91,227) .. (116.5,227) .. controls (120.09,227) and (123,229.91) .. (123,233.5) .. controls (123,237.09) and (120.09,240) .. (116.5,240) .. controls (112.91,240) and (110,237.09) .. (110,233.5) -- cycle ;
				\draw   (317.5,57) .. controls (317.5,53.41) and (320.41,50.5) .. (324,50.5) .. controls (327.59,50.5) and (330.5,53.41) .. (330.5,57) .. controls (330.5,60.59) and (327.59,63.5) .. (324,63.5) .. controls (320.41,63.5) and (317.5,60.59) .. (317.5,57) -- cycle ;
				\draw    (270.5,53.5) -- (365.5,320) ;
				\draw    (268.83,52.83) -- (170,314) ;
				\draw    (266.83,51.83) -- (121.33,229) ;
				\draw    (273.5,52.5) -- (423,221.67) ;
				\draw    (277,47) -- (317.5,55) ;
				\draw    (217.67,56) -- (264,47) ;
				\draw    (329,60.33) -- (364.33,79) ;
				\draw    (327,62.5) -- (405.33,277.33) ;
				\draw    (322,63.5) -- (321.5,344) ;
				\draw    (319,61.5) -- (137.5,275.5) ;
				\draw    (318,57.5) -- (121,174) ;
				\draw    (374.33,87.33) -- (398.33,114.33) ;
				\draw    (369.5,89.5) -- (370.5,319) ;
				\draw    (365.33,87.83) -- (268.5,353.5) ;
				\draw    (363.33,85.5) -- (123,233.5) ;
				\draw    (363,83) -- (140,127) ;
				\draw    (405.67,125.33) -- (422.5,166.5) ;
				\draw    (399.83,125.5) -- (326.5,345) ;
				\draw    (396.83,123.5) -- (213,342.67) ;
				\draw    (396,121.67) -- (122,178.67) ;
				\draw    (396,118) -- (172,89) ;
				\draw    (424.5,179.5) -- (425.5,221) ;
				\draw    (418.33,177.67) -- (273.33,355.67) ;
				\draw    (417,175) -- (173.33,317.67) ;
				\draw    (417,172) -- (139.33,130.67) ;
				\draw    (418.33,169) -- (218,59.33) ;
				\draw    (423.5,233) -- (409,277.67) ;
				\draw    (420.67,231) -- (217,345.67) ;
				\draw    (420,227.5) -- (140.5,280.5) ;
				\draw    (421,224.5) -- (169.33,93) ;
				\draw    (170,83) -- (205,60.33) ;
				\draw    (136.5,120.5) -- (160.33,92.33) ;
				\draw    (116.5,171.5) -- (129.5,131.5) ;
				\draw    (116.5,227) -- (115.5,184.5) ;
				\draw    (129.33,275.33) -- (118.17,239.5) ;
				\draw    (162.67,317) -- (137.17,286.17) ;
				\draw    (204,347.33) -- (172.83,325.5) ;
				\draw    (262,361.33) -- (215.83,352.5) ;
				\draw    (317.17,353.83) -- (274,363) ;
				\draw    (364.33,330) -- (328.17,348.17) ;
				\draw    (403.67,288.67) -- (374.5,321.83) ;
				\draw    (400,285.5) -- (174.5,321.5) ;
				\draw    (400.33,281.33) -- (122.67,236) ;
				\draw    (402.33,278.33) -- (216.33,63) ;
				\draw    (363.33,326) -- (140.33,284) ;
				\draw    (363.5,323) -- (121.33,181.33) ;
				\draw    (316.5,348) -- (120,238.5) ;
				\draw    (318.5,345) -- (135,133.67) ;
				\draw    (262,358) -- (119,184) ;
				\draw    (265,355) -- (167.33,95) ;
				\draw    (205.33,344) -- (132.5,133.5) ;
				\draw    (209.5,341.5) -- (211.33,64) ;
				\draw    (165,315) -- (163.33,94) ;
				\draw    (134,274) -- (207.33,62.67) ;
				
				\draw (266,23) node [anchor=north west][inner sep=0.75pt]   [align=left] {{\tiny 0}};
				\draw (328,33) node [anchor=north west][inner sep=0.75pt]   [align=left] {{\tiny 1}};
				\draw (377.5,61) node [anchor=north west][inner sep=0.75pt]   [align=left] {{\tiny 2}};
				\draw (413.5,104) node [anchor=north west][inner sep=0.75pt]   [align=left] {{\tiny 3}};
				\draw (433.5,162.5) node [anchor=north west][inner sep=0.75pt]   [align=left] {{\tiny 4}};
				\draw (437,224) node [anchor=north west][inner sep=0.75pt]   [align=left] {{\tiny 5}};
				\draw (414,286) node [anchor=north west][inner sep=0.75pt]   [align=left] {{\tiny 6}};
				\draw (376,331) node [anchor=north west][inner sep=0.75pt]   [align=left] {{\tiny 7}};
				\draw (200.5,34.5) node [anchor=north west][inner sep=0.75pt]   [align=left] {{\tiny 17}};
				\draw (324,360) node [anchor=north west][inner sep=0.75pt]   [align=left] {{\tiny 8}};
				\draw (266,372) node [anchor=north west][inner sep=0.75pt]   [align=left] {{\tiny 9}};
				\draw (198.67,360) node [anchor=north west][inner sep=0.75pt]   [align=left] {{\tiny 10}};
				\draw (153.33,330) node [anchor=north west][inner sep=0.75pt]   [align=left] {{\tiny 11}};
				\draw (108,276) node [anchor=north west][inner sep=0.75pt]   [align=left] {{\tiny 12}};
				\draw (91,228) node [anchor=north west][inner sep=0.75pt]   [align=left] {{\tiny 13}};
				\draw (91,166.33) node [anchor=north west][inner sep=0.75pt]   [align=left] {{\tiny 14}};
				\draw (148,67.67) node [anchor=north west][inner sep=0.75pt]   [align=left] {{\tiny 16}};
				\draw (114.67,111) node [anchor=north west][inner sep=0.75pt]   [align=left] {{\tiny 15}};

			\end{tikzpicture}
			\caption{$\uc(18)$ }
			\label{fig:sub2}
		\end{subfigure}
		\caption{Two examples of periodic unitary Cayley graphs}
		\label{fig}
	\end{figure}

	In the literature, very few classes of  graphs are found that are periodic with respect to the Grover walks. Examples include complete graphs, complete bipartite graphs, strongly regular graphs, generalized Bethe trees, Hamming graphs, and Johnson graphs. However, the periodicity of  Cayley graphs with respect to the Grover walks has not been explored in the past. This paper considers unitary Cayley graphs, which are a particular type of Cayley graph. For example, the unitary Cayley graphs on $12$ and $18$ vertices, shown in Figure \ref{fig}, are periodic.

	\section{Perfect state transfer on unitary Cayley graphs}
	Let $G$ be a simple graph. A bijective function $f: V(G)\rightarrow V(G)$ is said to be an \emph{automorphism} of $G$ if, $uv\in E(G)$ if and only if $f(u)f(v)\in E(G)$ for any two vertices $u$ and $v$. A graph $G$ is said to be \emph{vertex transitive}, if for any two vertices $u$ and $v$ in $G$, there exists an automorphism $f$ of $G$ such that $f(u)=v$. Note that a vertex-transitive graph is regular. Therefore the discriminant of the graph is $\frac{1}{k}A$, where $k$ is the regularity of the graph. We find that periodicity is necessary for the occurrence of perfect state transfer in vertex-transitive graphs.
	\begin{theorem}\label{thm2}
		Let $G$ be a vertex-transitive graph. If $G$ exhibits perfect state transfer, then it is periodic.
	\end{theorem}
	\begin{proof}
		Let  $G$ be a vertex-transitive graph. Let $E_r$ be the eigenprojectors of the discriminant $P$ of $G$ corresponding to an eigenvalue $\mu_r$ of $P$. Suppose $E_r\eu =\textbf{0}$ for some vertex $u$ of $G$. Since $G$ is vertex-transitive, $E_r\eu=\textbf{0}$ gives that $E_r\ev=\textbf{0}$ for each vertex $v$ of $G$. Therefore $E_r=\textbf{0}$, which is a contradiction. Thus $\Theta_P(u)=\spec_P(G)$ for any vertex $u$ of $G$. Let $\mu \in \spec_P(G)$. Since the graph exhibits perfect state transfer, using Theorem \ref{p1} and Lemma \ref{ch}, we have $\mu=\cos\frac{s}{\tau}\pi$ for some positive integers $s$ and $\tau$. Therefore, the eigenvalues of the time evolution matrix $U$ are of the form $e^{ \pm \frac{\iu s \pi}{\tau}}$ for some positive integers $s$ and $\tau$. Thus, there is a positive integer $k$ such that $U^k=I$. Hence $G$ is periodic.
	\end{proof}

	The next corollary is obtained from Theorem \ref{main1} and Theorem \ref{thm2}. 
	\begin{corollary}\label{coro1}
		If the unitary Cayley graph $\uc(n)$ exhibits perfect state transfer, then $n=2^\alpha 3^\beta$, for some non-negative integers $\alpha$ and $\beta$ with $\alpha +\beta \neq 0 $.
	\end{corollary}	
	
	\begin{theorem}\label{thl}
	Let $\mu_0,\mu_1,\hdots,\mu_{n-1}$  be the  eigenvalues of the discriminant of a circulant graph $\cay(\Zl_n, C)$, where $\mu_j=\frac{1}{|C|}\sum_{s\in C} e^{\frac{2\pi j s \iu }{n}}$ for $j\in\{0,\hdots, n-1\}$. Then prefect state transfer occurs in $\cay(\Zl_n, C)$ from a vertex $u$ to another vertex $v$ at time $\tau$ if and only if all of the following conditions hold.
		\begin{enumerate}[label=(\roman*)]
			\item $n$ is even and $u-v=\frac{n}{2}$.
			\item $T_\tau(\mu_j)=\pm1$ for $j\in\{0,\hdots,n-1\}$, where $T_n(x)$ is the Chebyshev polynomial of the first kind.
			\item $T_\tau(\mu_j)\neq T_\tau(\mu_{j+1})$ for $j\in\{0,\hdots,n-2\}.$
		\end{enumerate}
	\end{theorem}	
	\begin{proof}
		First we prove that, if Condition $(ii)$ holds, then perfect state transfer occurs in $\cay(\Zl_n,C)$  from a vertex $u$ to another vertex $v$ at time $\tau$ if and only if \begin{equation} \label{eqq}
			\frac{k_{j+1}-k_j}{2} + \frac{1}{n}(u-v)\in \Zl ~~\text{for}~~ j\in\{0,\hdots, n-2\},
		\end{equation} 
		where $k_j\in\{0,1\}$ is given by $T_{\tau}(\mu_j)=e^{\iu k_j\pi}$ for $j\in\{0,\hdots,n-1\}$.
		
		Assume that Condition $(ii)$ holds.  For $0\leq j\leq n-1$, define $k_j\in \{0,1\}$ such that $T_{\tau}(\mu_j)=e^{\iu k_j\pi}$. Now, perfect state transfer occurs in $\cay(\Zl_n,C)$  from a vertex $u$ to another vertex $v$ at time $\tau$ if and only if 
		\begin{align}
			1&=|\ip{U^\tau d^*\eu}{d^* \ev}| \nonumber \tag*{(by Lemma \ref{st11})}\\
			&=|\ip{dU^\tau d^*\eu}{\ev}| \nonumber\\
			&=|\ip{T_{\tau}(P)\eu}{\ev}| \nonumber\tag*{(by Lemma \ref{ch11})}\\
			&= |T_{\tau}(P)_{uv}|. \nonumber
		\end{align}
		Writing $|C|=r$, we have $P=\frac{1}{r}A$. Therefore from \eqref{evc}, the eigenvectors $\textbf{v}_j$  of $P$ corresponding to $\mu_j$ is given by  $\mathbf{v}_j=\left[1~~ \omega_n^j~~\omega_n^{2j}~~ ...~~ \omega_n^{(n-1)j}\right]^t$ for $j\in\{0, \hdots , n-1\}$. Thus, the spectral decomposition of $P$ is given by $$P=\frac{1}{n} \sum_{j=0}^{n-1} \mu_j  \mathbf{v}_j \mathbf{v}_j^*.$$
		Therefore 
  \begin{align*}
      T_{\tau}(P)&=\frac{1}{n} \sum_{j=0}^{n-1} T_{\tau}(\mu_j)  \mathbf{v}_j \mathbf{v}_j^* \tag*{(by  \eqref{sd})}\\
      &=\frac{1}{n} \sum_{j=0}^{n-1} e^{\iu k_j\pi}  \mathbf{v}_j \mathbf{v}_j^*.
  \end{align*}
   Hence the $uv$-th entry of $T_\tau(P)$ is given by
		\begin{equation}\label{eq}
			T_{\tau}(P)_{uv}=\frac{1}{n} \sum_{j=0}^{n-1} e^{\iu k_j\pi } \omega_n^{j(u-v)}.
		\end{equation}
		From the triangle inequality, the magnitude of the right side of \eqref{eq} is at most $1$. Therefore $|T_\tau(P)_{uv}|=1$ if and only if the summands in the right side of Equation \eqref{eq} are equal to each other. Thus $|T_\tau(P)_{uv}|=1$ if and only if \begin{equation*}
			e ^{\iu k_o \pi  }=e^{\iu k_1\pi  + \iu \frac{2\pi}{n}(u-v) }=\cdots  = e^{\iu k_{n-1}\pi + \iu \frac{2(n-1)\pi}{n}(u-v)}.\end{equation*}
		Now from $$e^{\iu k_j\pi + \iu\frac{2j\pi}{n}(u-v)}=e^{\iu k_{j+1}\pi+\iu\frac{2(j+1)\pi}{n}(u-v)}~\text{for}~j\in\{0,\hdots,n-2\},$$ we have 
  \begin{equation*} 
			\frac{k_{j+1}-k_j}{2} + \frac{1}{n}(u-v)\in \Zl ~~\text{for}~~ j\in\{0,\hdots, n-2\}.
		\end{equation*} 
		Thus if Condition $(ii)$ holds, then perfect state transfer occurs in $\cay(\Zl_n,C)$  from  $u$ to $v$ at time $\tau$  if and only if \eqref{eqq} holds.
		
		Now, assume that perfect state transfer occurs in $\cay(\Zl_n,C)$ from a vertex $u$ to another vertex $v$ at time $\tau$. Since  $\cay(\Zl_n,C)$ is a vertex-transitive graph, as in the proof of Theorem \ref{thm2}, we find that  $\Theta_P(u)=\spec_P(\cay(\Zl_n,C))$ for each vertex $u$. Therefore by Theorem \ref{p1}, we have $T_\tau(\mu_j)=\pm1$ for each $j\in\{0,\hdots,n-1\}$. Thus Condition $(ii)$ holds.

		 Observe that the eigenvalues $\mu_j$  satisfy 
		 $$\mu_j=\frac{1}{r}\sum_{s\in C}\omega_n^{js}=\frac{1}{r}\sum_{s\in C}\omega_n^{(n-j)s}=\mu_{n-j}~\text{for}~j\in\{1,\hdots, n-1\}.$$
		  If $n$ is  odd, then $\mu_\ell=\mu_{\ell+1}$ for $\ell=\frac{n-1}{2}$. This implies $T_{\tau}(\mu_\ell)=T_\tau(\mu_{\ell+1})$, and so $k_\ell=k_{\ell+1}$. Therefore from \eqref{eqq}, $\frac{1}{n}(u-v)\in \Zl$, which is a contradiction. So, $n$ must be even and $k_{j+1}\neq k_j$ for each  $j\in\{0,\hdots,n-2\}$. Since $k_j\in\{0,1\}$, we have $u-v=\pm \frac{n}{2} +kn$ for some integer $k$. Note that $u,v\in \Zl_n$ and $u\neq v$. Therefore $u-v= \frac{n}{2}$. Thus Condition $(i)$ holds.
		
		Finally suppose that $T_\tau(\mu_j)=T_\tau(\mu_{j+1})$ for some $j$. Then $k_j=k_{j+1}$, and so $\frac{1}{n}(u-v)\in \Zl$,  a contradiction. Therefore $T_\tau(\mu_j)\neq T_\tau(\mu_{j+1})$ for each $j\in\{0,\hdots,n-2\}.$ Thus Condition $(iii)$ holds.

		Conversely, assume that conditions $(i)$ to $(iii)$  hold.  Therefore $n$ is even and $u-v=\frac{n}{2}$.
		By Condition $(iii)$, $T_\tau(\mu_j)\neq T_\tau(\mu_{j+1})$, and so $k_{j}\neq k_{j+1}$ for each $j\in\{0,\hdots,n-2\}.$ Since $k_j\in\{0,1\}$, we find that $k_{j+1}-k_j\in\{\pm 1\}$. Thus Condition \eqref{eqq} holds. This, along with Condition $(ii)$, imply that perfect state transfer occurs in  $\cay(\Zl_n, C)$ from $u$ to $v$ at time $\tau$
\end{proof}
	\begin{corollary}\label{thl1}
		Let $\cay(\Zl_n,C)$ be a circulant graph with adjacency eigenvalues $\lambda_j=\sum_{s\in C} \omega_n^{js}$ for $j\in\{0,\hdots, n-1\}$, where $\omega_n=e^{\frac{2\pi \iu}{n}}$. If $\ld_j=\ld_{j+1}$ for some $j\in\{0, \hdots , n-2\}$, then perfect state transfer does not occur in $\cay(\Zl_n,C)$. 
	\end{corollary}
	\begin{proof}
		If $\ld_j=\ld_{j+1}$ for some $j\in\{0,\hdots,n-2\}$, then $\mu_j=\mu_{j+1}$. Hence $T_\tau(\mu_j)=T_\tau(\mu_{j+1})$. Therefore by Theorem \ref{thl}, the result follows.
	\end{proof}
	\begin{lema}\label{cycle}
		A cycle on $n$ vertices exhibit perfect state transfer if and only if $n$ is even.
	\end{lema}
	\begin{proof}
		A cycle on $n$ vertices is the circulant graph $\cay(\Zl_n, \{\pm 1\})$. Therefore, for $n$ odd, the cycle does not exhibit perfect state transfer. So, let $n=2m$ for some positive integer $m$ with $m\geq 2$. Then  the eigenvalue $\mu_j$ of the discriminant $P$ of $C_n$ is given by $$\mu_j=\frac{1}{2}(\omega_n^{j}+\omega_n^{-j})=\cos\left(\frac{2j\pi}{n}\right)=\cos\left(\frac{j\pi}{m} \right),$$ for $j\in\{0,\hdots, n-1\}$. Thus by \eqref{chb}, we have $T_m(\mu_j)=\cos(j\pi)$ for $j\in\{0,\hdots, n-1\}$. Therefore by Theorem \ref{thl}, perfect state transfer occurs in $C_n$ from the vertex $u$ to $u+\frac{n}{2}$ at time $\frac{n}{2}$ for even positive integers $n$ with $n\geq 4$.
	\end{proof}
	\begin{lema}\label{complete}
		The only complete graphs $K_n$ exhibiting perfect state transfer is $K_2$.
	\end{lema}
	\begin{proof}
		Note that $\spec_A(K_n)=\{ n-1,-1\}$. Therefore by Theorem \ref{thm1}, $K_n$ is periodic if and only if $n=2$ or $n=3$. Also $K_n=\cay(\Zl_n, \{1,\hdots,n-1\})$. Hence by Theorem \ref{thm2} and Theorem \ref{thl}, the result follows.
	\end{proof}
	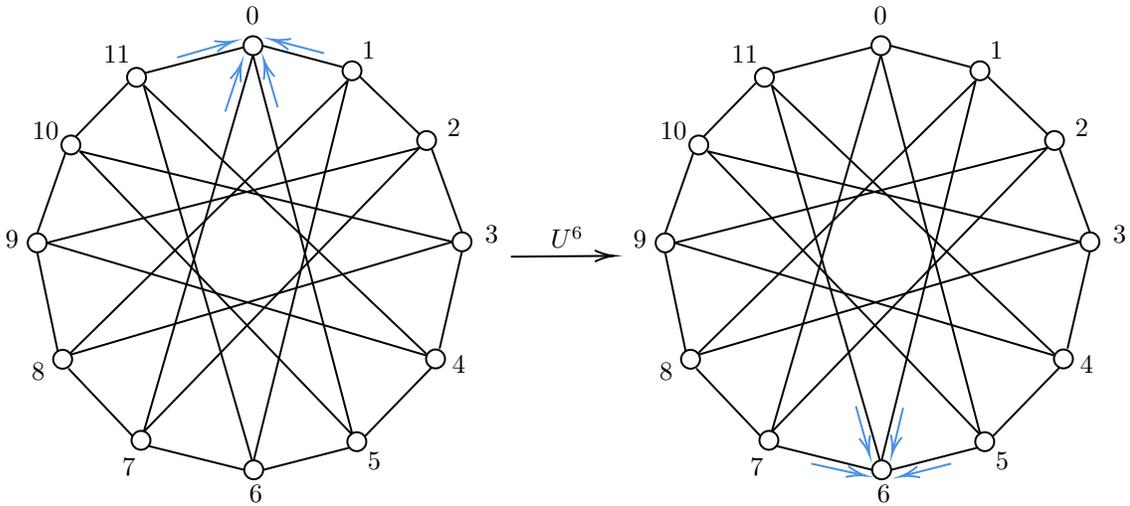
\begin{figure}[h!]
		\centering

		\tikzset{every picture/.style={line width=0.75pt}} 
		
		\begin{tikzpicture}[x=0.65pt,y=0.65pt,yscale=-1,xscale=1]
			
			\draw   (36.93,114.43) .. controls (36.93,111.4) and (39.4,108.93) .. (42.43,108.93) .. controls (45.47,108.93) and (47.93,111.4) .. (47.93,114.43) .. controls (47.93,117.47) and (45.47,119.93) .. (42.43,119.93) .. controls (39.4,119.93) and (36.93,117.47) .. (36.93,114.43) -- cycle ;
			\draw   (75.13,75.03) .. controls (75.13,72) and (77.6,69.53) .. (80.63,69.53) .. controls (83.67,69.53) and (86.13,72) .. (86.13,75.03) .. controls (86.13,78.07) and (83.67,80.53) .. (80.63,80.53) .. controls (77.6,80.53) and (75.13,78.07) .. (75.13,75.03) -- cycle ;
			\draw   (142.9,56.6) .. controls (142.9,53.56) and (145.36,51.1) .. (148.4,51.1) .. controls (151.44,51.1) and (153.9,53.56) .. (153.9,56.6) .. controls (153.9,59.64) and (151.44,62.1) .. (148.4,62.1) .. controls (145.36,62.1) and (142.9,59.64) .. (142.9,56.6) -- cycle ;
			\draw   (249.13,239.03) .. controls (249.13,236) and (251.6,233.53) .. (254.63,233.53) .. controls (257.67,233.53) and (260.13,236) .. (260.13,239.03) .. controls (260.13,242.07) and (257.67,244.53) .. (254.63,244.53) .. controls (251.6,244.53) and (249.13,242.07) .. (249.13,239.03) -- cycle ;
			\draw   (17.33,171.43) .. controls (17.33,168.4) and (19.8,165.93) .. (22.83,165.93) .. controls (25.87,165.93) and (28.33,168.4) .. (28.33,171.43) .. controls (28.33,174.47) and (25.87,176.93) .. (22.83,176.93) .. controls (19.8,176.93) and (17.33,174.47) .. (17.33,171.43) -- cycle ;
			\draw   (32.13,239.23) .. controls (32.13,236.2) and (34.6,233.73) .. (37.63,233.73) .. controls (40.67,233.73) and (43.13,236.2) .. (43.13,239.23) .. controls (43.13,242.27) and (40.67,244.73) .. (37.63,244.73) .. controls (34.6,244.73) and (32.13,242.27) .. (32.13,239.23) -- cycle ;
			\draw   (200.53,71.23) .. controls (200.53,68.2) and (203,65.73) .. (206.03,65.73) .. controls (209.07,65.73) and (211.53,68.2) .. (211.53,71.23) .. controls (211.53,74.27) and (209.07,76.73) .. (206.03,76.73) .. controls (203,76.73) and (200.53,74.27) .. (200.53,71.23) -- cycle ;
			\draw   (243.93,111.83) .. controls (243.93,108.8) and (246.4,106.33) .. (249.43,106.33) .. controls (252.47,106.33) and (254.93,108.8) .. (254.93,111.83) .. controls (254.93,114.87) and (252.47,117.33) .. (249.43,117.33) .. controls (246.4,117.33) and (243.93,114.87) .. (243.93,111.83) -- cycle ;
			\draw   (264.53,170.83) .. controls (264.53,167.8) and (267,165.33) .. (270.03,165.33) .. controls (273.07,165.33) and (275.53,167.8) .. (275.53,170.83) .. controls (275.53,173.87) and (273.07,176.33) .. (270.03,176.33) .. controls (267,176.33) and (264.53,173.87) .. (264.53,170.83) -- cycle ;
			\draw   (143.33,303.63) .. controls (143.33,300.6) and (145.8,298.13) .. (148.83,298.13) .. controls (151.87,298.13) and (154.33,300.6) .. (154.33,303.63) .. controls (154.33,306.67) and (151.87,309.13) .. (148.83,309.13) .. controls (145.8,309.13) and (143.33,306.67) .. (143.33,303.63) -- cycle ;
			\draw   (203.33,287.23) .. controls (203.33,284.2) and (205.8,281.73) .. (208.83,281.73) .. controls (211.87,281.73) and (214.33,284.2) .. (214.33,287.23) .. controls (214.33,290.27) and (211.87,292.73) .. (208.83,292.73) .. controls (205.8,292.73) and (203.33,290.27) .. (203.33,287.23) -- cycle ;
			\draw   (77.73,286.43) .. controls (77.73,283.4) and (80.2,280.93) .. (83.23,280.93) .. controls (86.27,280.93) and (88.73,283.4) .. (88.73,286.43) .. controls (88.73,289.47) and (86.27,291.93) .. (83.23,291.93) .. controls (80.2,291.93) and (77.73,289.47) .. (77.73,286.43) -- cycle ;
			\draw    (86,72.4) -- (142.9,57.6) ;
			\draw    (45.4,109.4) -- (75.6,78.4) ;
			\draw    (154.33,303.63) -- (204.23,290.83) ;
			\draw    (153.9,56.6) -- (200.4,69) ;
			\draw    (245.6,108.4) -- (210.5,74.6) ;
			\draw    (22.83,165.93) -- (39.3,119) ;
			\draw    (22.83,176.93) -- (34.8,234.6) ;
			\draw    (257,233) -- (270.03,176.33) ;
			\draw    (41,244.6) -- (78.2,283.8) ;
			\draw    (87.8,289.8) -- (143.33,303.63) ;
			\draw    (270.03,165.33) -- (252.5,117) ;
			\draw    (213.8,283.8) -- (252.2,244.2) ;
			\draw    (148.4,62.1) -- (206,282) ;
			\draw    (148.4,62.1) -- (84.8,281.6) ;
			\draw    (203.6,76.9) -- (148.83,298.13) ;
			\draw    (203.6,76.9) -- (42.2,235.6) ;
			\draw    (245.2,116.1) -- (84.8,281.6) ;
			\draw    (245.2,116.1) -- (28.33,171.43) ;
			\draw    (264.53,170.83) -- (42.2,236.6) ;
			\draw    (47.6,118.1) -- (264.53,170.83) ;
			\draw    (28.33,171.43) -- (249.2,237) ;
			\draw    (84.8,79.7) -- (249.2,236) ;
			\draw    (47.6,118.1) -- (198.84,275.59) -- (206,283) ;
			\draw    (84.8,79.7) -- (147.83,298.13) ;
			\draw   (402.33,114.43) .. controls (402.33,111.4) and (404.8,108.93) .. (407.83,108.93) .. controls (410.87,108.93) and (413.33,111.4) .. (413.33,114.43) .. controls (413.33,117.47) and (410.87,119.93) .. (407.83,119.93) .. controls (404.8,119.93) and (402.33,117.47) .. (402.33,114.43) -- cycle ;
			\draw   (440.53,75.03) .. controls (440.53,72) and (443,69.53) .. (446.03,69.53) .. controls (449.07,69.53) and (451.53,72) .. (451.53,75.03) .. controls (451.53,78.07) and (449.07,80.53) .. (446.03,80.53) .. controls (443,80.53) and (440.53,78.07) .. (440.53,75.03) -- cycle ;
			\draw   (508.3,56.6) .. controls (508.3,53.56) and (510.76,51.1) .. (513.8,51.1) .. controls (516.84,51.1) and (519.3,53.56) .. (519.3,56.6) .. controls (519.3,59.64) and (516.84,62.1) .. (513.8,62.1) .. controls (510.76,62.1) and (508.3,59.64) .. (508.3,56.6) -- cycle ;
			\draw   (614.53,239.03) .. controls (614.53,236) and (617,233.53) .. (620.03,233.53) .. controls (623.07,233.53) and (625.53,236) .. (625.53,239.03) .. controls (625.53,242.07) and (623.07,244.53) .. (620.03,244.53) .. controls (617,244.53) and (614.53,242.07) .. (614.53,239.03) -- cycle ;
			\draw   (382.73,171.43) .. controls (382.73,168.4) and (385.2,165.93) .. (388.23,165.93) .. controls (391.27,165.93) and (393.73,168.4) .. (393.73,171.43) .. controls (393.73,174.47) and (391.27,176.93) .. (388.23,176.93) .. controls (385.2,176.93) and (382.73,174.47) .. (382.73,171.43) -- cycle ;
			\draw   (397.53,239.23) .. controls (397.53,236.2) and (400,233.73) .. (403.03,233.73) .. controls (406.07,233.73) and (408.53,236.2) .. (408.53,239.23) .. controls (408.53,242.27) and (406.07,244.73) .. (403.03,244.73) .. controls (400,244.73) and (397.53,242.27) .. (397.53,239.23) -- cycle ;
			\draw   (565.93,71.23) .. controls (565.93,68.2) and (568.4,65.73) .. (571.43,65.73) .. controls (574.47,65.73) and (576.93,68.2) .. (576.93,71.23) .. controls (576.93,74.27) and (574.47,76.73) .. (571.43,76.73) .. controls (568.4,76.73) and (565.93,74.27) .. (565.93,71.23) -- cycle ;
			\draw   (609.33,111.83) .. controls (609.33,108.8) and (611.8,106.33) .. (614.83,106.33) .. controls (617.87,106.33) and (620.33,108.8) .. (620.33,111.83) .. controls (620.33,114.87) and (617.87,117.33) .. (614.83,117.33) .. controls (611.8,117.33) and (609.33,114.87) .. (609.33,111.83) -- cycle ;
			\draw   (629.93,170.83) .. controls (629.93,167.8) and (632.4,165.33) .. (635.43,165.33) .. controls (638.47,165.33) and (640.93,167.8) .. (640.93,170.83) .. controls (640.93,173.87) and (638.47,176.33) .. (635.43,176.33) .. controls (632.4,176.33) and (629.93,173.87) .. (629.93,170.83) -- cycle ;
			\draw   (508.73,303.63) .. controls (508.73,300.6) and (511.2,298.13) .. (514.23,298.13) .. controls (517.27,298.13) and (519.73,300.6) .. (519.73,303.63) .. controls (519.73,306.67) and (517.27,309.13) .. (514.23,309.13) .. controls (511.2,309.13) and (508.73,306.67) .. (508.73,303.63) -- cycle ;
			\draw   (568.73,287.23) .. controls (568.73,284.2) and (571.2,281.73) .. (574.23,281.73) .. controls (577.27,281.73) and (579.73,284.2) .. (579.73,287.23) .. controls (579.73,290.27) and (577.27,292.73) .. (574.23,292.73) .. controls (571.2,292.73) and (568.73,290.27) .. (568.73,287.23) -- cycle ;
			\draw   (443.13,286.43) .. controls (443.13,283.4) and (445.6,280.93) .. (448.63,280.93) .. controls (451.67,280.93) and (454.13,283.4) .. (454.13,286.43) .. controls (454.13,289.47) and (451.67,291.93) .. (448.63,291.93) .. controls (445.6,291.93) and (443.13,289.47) .. (443.13,286.43) -- cycle ;
			\draw    (451.4,72.4) -- (508.3,57.6) ;
			\draw    (410.8,109.4) -- (441,78.4) ;
			\draw    (519.73,303.63) -- (569.63,290.83) ;
			\draw    (519.3,56.6) -- (565.8,69) ;
			\draw    (611,108.4) -- (575.9,74.6) ;
			\draw    (388.23,165.93) -- (404.7,119) ;
			\draw    (388.23,176.93) -- (400.2,234.6) ;
			\draw    (622.4,233) -- (635.43,176.33) ;
			\draw    (406.4,244.6) -- (443.6,283.8) ;
			\draw    (453.2,289.8) -- (508.73,303.63) ;
			\draw    (635.43,165.33) -- (617.9,117) ;
			\draw    (579.2,283.8) -- (617.6,244.2) ;
			\draw    (513.8,62.1) -- (571.4,282) ;
			\draw    (513.8,62.1) -- (450.2,281.6) ;
			\draw    (569,76.9) -- (514.23,298.13) ;
			\draw    (569,76.9) -- (407.6,235.6) ;
			\draw    (610.6,116.1) -- (450.2,281.6) ;
			\draw    (610.6,116.1) -- (393.73,171.43) ;
			\draw    (629.93,170.83) -- (407.6,236.6) ;
			\draw    (413,118.1) -- (629.93,170.83) ;
			\draw    (393.73,171.43) -- (614.6,237) ;
			\draw    (450.2,79.7) -- (614.6,236) ;
			\draw    (413,118.1) -- (564.24,275.59) -- (571.4,283) ;
			\draw    (450.2,79.7) -- (513.23,298.13) ;
			\draw    (298.5,179.5) -- (356.25,179.02) ;
			\draw [shift={(358.25,179)}, rotate = 179.52] [color={rgb, 255:red, 0; green, 0; blue, 0 }  ][line width=0.75]    (10.93,-3.29) .. controls (6.95,-1.4) and (3.31,-0.3) .. (0,0) .. controls (3.31,0.3) and (6.95,1.4) .. (10.93,3.29)   ;
			\draw [color={rgb, 255:red, 74; green, 144; blue, 226 }  ,draw opacity=1 ]   (104.25,63.5) -- (135.33,54.55) ;
			\draw [shift={(137.25,54)}, rotate = 163.94] [color={rgb, 255:red, 74; green, 144; blue, 226 }  ,draw opacity=1 ][line width=0.75]    (10.93,-3.29) .. controls (6.95,-1.4) and (3.31,-0.3) .. (0,0) .. controls (3.31,0.3) and (6.95,1.4) .. (10.93,3.29)   ;
			\draw [color={rgb, 255:red, 74; green, 144; blue, 226 }  ,draw opacity=1 ]   (189.75,61) -- (161.18,53.51) ;
			\draw [shift={(159.25,53)}, rotate = 14.7] [color={rgb, 255:red, 74; green, 144; blue, 226 }  ,draw opacity=1 ][line width=0.75]    (10.93,-3.29) .. controls (6.95,-1.4) and (3.31,-0.3) .. (0,0) .. controls (3.31,0.3) and (6.95,1.4) .. (10.93,3.29)   ;
			\draw [color={rgb, 255:red, 74; green, 144; blue, 226 }  ,draw opacity=1 ]   (132.25,95) -- (141.12,68.4) ;
			\draw [shift={(141.75,66.5)}, rotate = 108.43] [color={rgb, 255:red, 74; green, 144; blue, 226 }  ,draw opacity=1 ][line width=0.75]    (10.93,-3.29) .. controls (6.95,-1.4) and (3.31,-0.3) .. (0,0) .. controls (3.31,0.3) and (6.95,1.4) .. (10.93,3.29)   ;
			\draw [color={rgb, 255:red, 74; green, 144; blue, 226 }  ,draw opacity=1 ]   (162.75,92.5) -- (155.31,66.92) ;
			\draw [shift={(154.75,65)}, rotate = 73.78] [color={rgb, 255:red, 74; green, 144; blue, 226 }  ,draw opacity=1 ][line width=0.75]    (10.93,-3.29) .. controls (6.95,-1.4) and (3.31,-0.3) .. (0,0) .. controls (3.31,0.3) and (6.95,1.4) .. (10.93,3.29)   ;
			\draw [color={rgb, 255:red, 74; green, 144; blue, 226 }  ,draw opacity=1 ]   (473.25,300) -- (502.8,306.57) ;
			\draw [shift={(504.75,307)}, rotate = 192.53] [color={rgb, 255:red, 74; green, 144; blue, 226 }  ,draw opacity=1 ][line width=0.75]    (10.93,-3.29) .. controls (6.95,-1.4) and (3.31,-0.3) .. (0,0) .. controls (3.31,0.3) and (6.95,1.4) .. (10.93,3.29)   ;
			\draw [color={rgb, 255:red, 74; green, 144; blue, 226 }  ,draw opacity=1 ]   (554.75,300) -- (527.19,307.01) ;
			\draw [shift={(525.25,307.5)}, rotate = 345.74] [color={rgb, 255:red, 74; green, 144; blue, 226 }  ,draw opacity=1 ][line width=0.75]    (10.93,-3.29) .. controls (6.95,-1.4) and (3.31,-0.3) .. (0,0) .. controls (3.31,0.3) and (6.95,1.4) .. (10.93,3.29)   ;
			\draw [color={rgb, 255:red, 74; green, 144; blue, 226 }  ,draw opacity=1 ]   (526.75,267.5) -- (520.72,292.56) ;
			\draw [shift={(520.25,294.5)}, rotate = 283.54] [color={rgb, 255:red, 74; green, 144; blue, 226 }  ,draw opacity=1 ][line width=0.75]    (10.93,-3.29) .. controls (6.95,-1.4) and (3.31,-0.3) .. (0,0) .. controls (3.31,0.3) and (6.95,1.4) .. (10.93,3.29)   ;
			\draw [color={rgb, 255:red, 74; green, 144; blue, 226 }  ,draw opacity=1 ]   (499.25,266.5) -- (506.71,293.07) ;
			\draw [shift={(507.25,295)}, rotate = 254.32] [color={rgb, 255:red, 74; green, 144; blue, 226 }  ,draw opacity=1 ][line width=0.75]    (10.93,-3.29) .. controls (6.95,-1.4) and (3.31,-0.3) .. (0,0) .. controls (3.31,0.3) and (6.95,1.4) .. (10.93,3.29)   ;
			
			\draw (142.6,32.2) node [anchor=north west][inner sep=0.75pt]   [align=left] {0};
			\draw (210.2,52.6) node [anchor=north west][inner sep=0.75pt]   [align=left] {1};
			\draw (259.8,97.4) node [anchor=north west][inner sep=0.75pt]   [align=left] {2};
			\draw (262.6,235.4) node [anchor=north west][inner sep=0.75pt]   [align=left] {4};
			\draw (281.8,159.8) node [anchor=north west][inner sep=0.75pt]   [align=left] {3};
			\draw (213.4,291.4) node [anchor=north west][inner sep=0.75pt]   [align=left] {5};
			\draw (144.6,310.6) node [anchor=north west][inner sep=0.75pt]   [align=left] {6};
			\draw (70.6,295) node [anchor=north west][inner sep=0.75pt]   [align=left] {7};
			\draw (17.8,239.8) node [anchor=north west][inner sep=0.75pt]   [align=left] {8};
			\draw (2.6,162.6) node [anchor=north west][inner sep=0.75pt]   [align=left] {9};
			\draw (18.2,99.4) node [anchor=north west][inner sep=0.75pt]   [align=left] {10};
			\draw (59.8,54.6) node [anchor=north west][inner sep=0.75pt]   [align=left] {11};
			\draw (508,32.2) node [anchor=north west][inner sep=0.75pt]   [align=left] {0};
			\draw (575.6,52.6) node [anchor=north west][inner sep=0.75pt]   [align=left] {1};
			\draw (625.2,97.4) node [anchor=north west][inner sep=0.75pt]   [align=left] {2};
			\draw (628,235.4) node [anchor=north west][inner sep=0.75pt]   [align=left] {4};
			\draw (647.2,159.8) node [anchor=north west][inner sep=0.75pt]   [align=left] {3};
			\draw (578.8,291.4) node [anchor=north west][inner sep=0.75pt]   [align=left] {5};
			\draw (510,310.6) node [anchor=north west][inner sep=0.75pt]   [align=left] {6};
			\draw (436,295) node [anchor=north west][inner sep=0.75pt]   [align=left] {7};
			\draw (383.2,239.8) node [anchor=north west][inner sep=0.75pt]   [align=left] {8};
			\draw (368,162.6) node [anchor=north west][inner sep=0.75pt]   [align=left] {9};
			\draw (383.6,99.4) node [anchor=north west][inner sep=0.75pt]   [align=left] {10};
			\draw (425.2,54.6) node [anchor=north west][inner sep=0.75pt]   [align=left] {11};
			\draw (320,160) node [anchor=north west][inner sep=0.75pt]  [font=\normalsize] [align=left] { $U^6$};
		\end{tikzpicture}
		\caption{Perfect state transfer in $\uc(12)$ from the vertex $0$ to $6$ at time $6$}
		\label{c6}
	\end{figure}
	
	\begin{theorem}
		The only unitary Cayley graphs $\uc(n)$ exhibiting perfect state transfer are $K_2,~C_4,~C_6$ and $\uc(12)$. 
	\end{theorem}
	\begin{proof}
		If perfect state transfer occurs in $\uc(n)$, then by Corollary  \ref{coro1}, $n=2^\alpha 3^\beta$, where $\alpha ~\text{and}~ \beta $ are non-negative integers with $\alpha + \beta \neq 0 $. For $n= 2,~ 4 $ and $6$, the graphs $\uc(n)$ are  $K_2,~C_4$ and $C_6$, respectively. From Lemma \ref{complete} and Lemma \ref{cycle}, these graphs exhibit perfect state transfer.
		
		For $n=12$, we find from the proof of Lemma \ref{lastuclemma} that $\spec_P(\uc(12))=\{\pm 1,\pm\frac{1}{2},0\}$. Therefore the period of $\uc(12)$ is $12$. Suppose $\uc(12)$ exhibits perfect state transfer at time $\tau$. Since the period of $\uc(12)$ is $12$, we must have  $\tau\in\{1,2,\hdots,11\}$. Now $0,-\frac{1}{2}\in \spec_P(\uc(12))$, and therefore by Theorem \ref{thl}, $T_\tau(0)=\pm 1$ and $T_\tau(-\frac{1}{2})=\pm 1 $, that is, $T_\tau(\cos\frac{\pi}{2})=\pm1$ and $T_\tau(\cos \frac{2\pi}{3})=\pm 1$. Hence by \eqref{chb}, $\cos(\frac{\tau \pi}{2})=\pm 1$ and $\cos(\frac{2\tau \pi}{3})=\pm1$. Therefore $\tau\in 2\Zl$ and $\tau\in3\Zl$, and so $\tau\in6\Zl$. Thus $\tau=6$. Now it is easy to see that $T_6(\mu_j)=\pm1$ for $j\in\{0,\hdots,11\}$ and $T_6(\mu_j)\neq T_6(\mu_{j+1})$ for $j\in\{0,\hdots,10\}$. Therefore by Theorem \ref{thl}, $\uc(12)$ exhibits perfect state transfer from the vertex $u$ to the vertex $u+6$ at time $6$.

		By Theorem \ref{thl}, perfect state transfer does not occur in $\uc(n)$ for $n=3^{\beta}$. 
		
		For $n=2^{\alpha}~ (\alpha \geq 3)$ or $n=2^{\alpha}3~ (\alpha\geq 3$) or $n=3^\beta 2~(\beta \geq 2)$ or $n=2^\alpha3^\beta~(\alpha\geq 2~\text{and}~\beta\geq 2)$, we have from Theorem \ref{main} that $\ld_1 =0=\ld_2$. 
		Hence by Corollary \ref{thl1}, perfect state transfer does not occur in these cases either. Therefore the result follows.
	\end{proof}

	\section{Periodic integral regular graphs}\label{org}
	
	Every path and cycle on $n$ vertices are periodic, see \cite{mixedpaths}. Kubota \cite{bipartite} proved that if a bipartite regular graph with four distinct adjacency eigenvalues is periodic, then it is $C_6$.
	
	\begin{theorem}\cite{bipartite} \label{sho} Let $G$ be a bipartite regular graph with four distinct adjacency eigenvalues. Then $G$ is periodic if and only if $G$ is isomorphic to $C_6$. \end{theorem}
	In this section, we give a spectral characterization of integral regular periodic graphs. The following theorem is due to Hoffman.
	\begin{theorem} \cite{polynomial} \label{polynomial}
		Let $G$ be a connected k-regular graph with n vertices, and let $k>\lambda_2 >\lambda_3 > ... >\lambda_r$ be the distinct eigenvalues of the adjacency matrix $A$ of $G$. Then $q(A)=\frac{q(k)}{n} {J}$, where $q(x)=\prod_{i=2}^{r}(x-\lambda_i)$.	
	\end{theorem}
	Let $G$  be a $k$-regular graph with $n$ vertices that are neither complete nor empty. Then $G$ is called \emph{strongly regular} with parameters $(n, k, a, c)$, denoted $\srg(n,k, a,c)$, if any two adjacent vertices of $G$ have $a$ common neighbors and any two non-adjacent vertices of $G$ have $c$ common neighbors.
	
	A graph $G$ is \emph{walk-regular} if its vertex-deleted subgraphs $G - u$ are cospectral for all $u\in V(G)$. For more information about walk-regular graphs, see \cite{walk}. The following result tells when a graph is walk-regular.
	\begin{theorem}\cite{walk}\label{walkr}
		A graph $G$ with adjacency matrix $A$ is walk-regular if and only if the diagonal entries of $A^r$ are constant for each non-negative integer $r$.
	\end{theorem}
	The following result uses some ideas from \cite{vandam}.
	\begin{lema}
		Let $G$ be a regular integral graph. If $G$ is periodic, then it is walk-regular.
	\end{lema}
	\begin{proof}
		Let $G$ be a $k$-regular integral graph. Then $G$ is periodic if and only if $\spec_A(G)\subseteq\{\pm k, \pm \frac{k}{2}, 0\}$.  
		
		If $\spec_A(G)=\{\pm k, \pm \frac{k}{2}, 0\}$, then $G$ is bipartite, and hence $(A^3)_{uu}=0$ for each $u\in V(G)$. Let $r$ be a non-negative integer. We now use Theorem \ref{polynomial} to prove that $A^r$ is a rational linear combination of $I,J,A^2$ and $A^3$. Note that $q(x)$ is a monic polynomial of degree $4$ with rational coefficients. Therefore $q(A)=\frac{q(k)}{n} {J}$ gives that $A^4$ is a rational linear combination of $I,J,A,A^2$ and $A^3$. Thus we find that $A^r=c_1I+c_2J+c_3A+c_4A^2+c_5A^3$ for some $c_1, c_2, c_3, c_4, c_5\in \Ql$. Therefore $(A^r)_{uu}=c_1+c_2+kc_4$ for each $u\in V(G)$. Hence by Theorem \ref{walkr}, $G$ is walk-regular.
		
		Now assume that $\spec_A(G)$ is a proper subset of $\{\pm k, \pm \frac{k}{2}, 0\}$. In this case, $q(x)$ is a monic polynomial of degree at most $3$ with rational coefficients. As in the previous paragraph, $A^r$ is a rational linear combination of $I,J,A$ and $A^2$ for each non-negative integer $r$. Therefore $A^r$ has constant diagonal for each non-negative integer $r$. Hence $G$ is walk-regular.
	\end{proof} 
	\begin{lema}\label{last1}
		Let $k$ be an odd positive integer and let $G$ be $k$-regular integral graph. Then $G$ is periodic if and only if it is $K_{k,k}$.
	\end{lema}
	\begin{proof}
		Let $k$ be an odd positive integer and $G$ be a $k$-regular integral graph. Thus,  $\pm\frac{k}{2}\notin\spec_A(G)$, and therefore by Theorem \ref{thm1}, $G$ is periodic if and only if $\spec_A(G)\subseteq \{\pm k, 0\}$. Since $\text{Trace}(A)=0$, either $\spec_A(G)=\{\pm k\}$ or $\spec_A(G)=\{\pm k, 0\}$.

		\textbf{Case 1.} $\spec_A(G)=\{\pm k\}$. By Theorem \ref{polynomial}, we have $$A+kI=\frac{2k}{n}J.$$ From this equation, we find that $k=1$ and $n=2$. So, the graph must be $K_{1,1}$.
		
		\textbf{Case 2.} $\spec_A(G)=\{\pm k, 0\}$. It is a strongly regular graph since $G$ has three distinct eigenvalues. Let $(n,k,a,c)$ be the parameter of $G$. By Theorem \ref{polynomial}, we have \begin{equation}\label{eqa2}    
			A^2+kA=\frac{2k^2}{n}J.\end{equation} 
		Equating a diagonal entry from both sides of \eqref{eqa2}, we find that $n=2k$. Let $u$ and $v$ be two adjacent vertices of $G$. Equating the $uv$-th entry from both sides of \eqref{eqa2}, we have $a=0$. Similarly, for non-adjacent vertices, we find $c=k$. Therefore $G$ is $\srg(2k,k,0,k)$. Since $-k\in \spec_A(G)$, the graph is bipartite. It is easy to see that a bipartite $\srg(2k,k,0,k)$ is indeed $K_{k,k}$.
	\end{proof}

	\begin{lema}
		The only $2$-regular integral periodic graphs are $C_3$, $C_4$ and $C_6$.
	\end{lema}
	\begin{proof}
		Note that connected 2-regular graphs are the cycles $C_n$ for $n\geq3$. It is also well known that the only integral cycles are $C_3$, $C_4$ and $C_6$. Hence the result follows.
	\end{proof} 
	\begin{lema}\label{hgg}
		Let $m$ be a positive integer with $m\geq 2$ and $k=2m$. Let $G$ be a k-regular integral non-bipartite graph. Then $G$ is periodic if and only if
		either $G=K_{\frac{k}{2},\frac{k}{2},\frac{k}{2}}$ or $\spec_A(G)=\{ k, \pm \frac{k}{2} , 0\}$.
	\end{lema}
	\begin{proof}
		Let $G$ be a $k$-regular integral non-bipartite graph, where $k=2m$ $(m\geq 2)$. Note that by Theorem \ref{thm1}, $G$ is periodic if and only if $\spec_A(G)\subseteq \{ \pm 2m, \pm m, 0\}$. Since $G$ is non-bipartite, $-2m\notin \spec_A(G)$. Thus $G$ is periodic if and only if $\spec_A(G)\subseteq\{2m,\pm m,0\}$. Note that Trace$(A)=0$. Therefore four possible cases arise.

		\textbf{Case 1.} $\spec_A(G)= \{ 2m,  -m\}$. By Theorem \ref{polynomial}, we have $$A+mI=\frac{3m}{n}J.$$ We easily get a contradiction by equating entries of both sides of this equation.
		
		\textbf{Case 2.} $\spec_A(G)= \{ 2m,  -m, 0\}$. In this case, $G$ is strongly regular, and let $G=\srg(n,2m, a,c)$. By Theorem \ref{polynomial}, we have \begin{equation*}
			A^2+mA=\frac{6m^2}{n}J.\end{equation*} Equating entries of both sides of this equation, we find that $n=3m, a=m$ and $c=2m$. Therefore the graph is $K_{m,m,m}$.
		
		\textbf{Case 3.} $\spec_A(G)= \{ 2m,  \pm m\}$. By Theorem \ref{polynomial}, we have $$A^2-m^2I=\frac{3m^2}{n}J.$$ Equating the diagonal entries of both sides of this equation, we have $2-m=\frac{3m}{n}$, which does not hold true for $m\geq 2$.
		
		\textbf{Case 4.} $\spec_A(G)= \{ 2m, \pm m, 0\}$. In this case, $G$ is periodic.
		
		Combining all the cases, the result follows.
	\end{proof}

	\begin{lema}\label{hgg2}
		Let $m$ be a positive integer with $m\geq 2$ and let $k=2m$. Let $G$ be a $k$-regular integral bipartite graph. Then $G$ is periodic if and only if
		either $G=K_{k,k}$ or $\spec_A(G)=\{\pm k, \pm \frac{k}{2} , 0\}$.
	\end{lema}
	\begin{proof}
		Let $G$ be a $k$-regular integral bipartite graph, where $k=2m\ (m\geq 2)$. Note that by Theorem \ref{thm1}, $G$ is periodic if and only if $\spec_A(G)\subseteq \{\pm 2m, \pm m , 0\}$. Since Trace$(A)=0$ and the graph is bipartite, four cases arise. 
		
		\textbf{Case 1.}  $\spec_A(G)=\{\pm 2m\}$. Clearly $G$ is not complete. By Theorem \ref{polynomial}, we have \begin{equation*}
			A+2mI=\frac{4m}{n}J.\end{equation*} Let $u$ and $v$ be two non-adjacent vertices of $G$. Equating  $uv$-th entry of both sides of the previous equation, we have $m=0$, a contradiction.
		
		\textbf{Case 2.} $\spec_A(G)=\{\pm 2m, 0\}$. Since the graph has three distinct eigenvalues, it is strongly regular, and let $G=\srg(n,2m, a,c)$. By Theorem \ref{polynomial}, we have $$A^2+2mA=\frac{8m^2}{n}J.$$ Equating entries of both sides of the preceding equation, it is easy to find that $n=4m,\ a=0$ and $c=2m$. A bipartite $\srg(4m,2m,0,2m)$ is indeed $K_{k,k}$.
		
		\textbf{Case 3.} $\spec_A(G)=\{\pm 2m, \pm m\}$. By Theorem \ref{sho}, it follows that the graph is $C_6$, which is a $2$-regular graph. But we considered that $G$ is $2m$-regular, where $m\geq 2$. Therefore, this case is not possible.

		\textbf{Case 4.} $\spec_A(G)=\{\pm 2m, \pm m , 0\}$. In this case, the graph is periodic.
		
		Combining all the cases, the result follows.
	\end{proof} 
	Lemma \ref{last1} to Lemma \ref{hgg2} altogether give the following theorem.
	\begin{theorem}
		A graph $G$ is $k$-regular, integral and periodic if and only if it is either $C_6$ or complete bipartite or complete tripartite or $\spec_A(G)=\{k, \pm \frac{k}{2},0\}$ or $\spec_A(G)=\{\pm k,\pm \frac{k}{2},0\}$.		
	\end{theorem}

	Let $F$ be a set of $t$ elements and $s$ be a positive integer. The \emph{Hamming graph} $H(s,t)$ has the vertex set $F^s$, the set of $s$-tuples of elements of $F$, and two vertices are adjacent if they differ by precisely at one co-ordinate. It is well known that $H(s,t)$ is $s(t-1)$-regular, and it is easy to calculate that $\spec_A(H(3,3))=\{6,\pm 3,0\}$. Thus $H(3,3)$ is a non-bipartite graph satisfying the condition of Case 4 in Lemma \ref{hgg}.

	Also from the proof of Lemma \ref{lastuclemma}, we find that $\spec_A(\uc(n)=\{\pm\varphi(n),\pm\frac{\varphi(n)}{2},0\}$ for $n=2^\alpha3^\beta$ with $\alpha\geq 1$ and $\beta\geq1$. Considering $m$ and $n$ such that $\varphi(n)=2m$, we see that $\uc(n)$ is a bipartite graph satisfying the condition of Case 4 in Lemma \ref{hgg2}.

	Thus we see that there are $k$-regular integral periodic graphs whose adjacency spectrum is either $\{ k,\pm \frac{k}{2},0\}$ or $\{\pm k,\pm \frac{k}{2},0\}$ for some even integer $k$. It is an interesting problem to characterize all such $k$-regular integral periodic graphs.

	\subsection*{Acknowledgments}
	The first author acknowledges the support provided by the Prime Minister’s Research Fellowship (PMRF) scheme of the Government of India (PMRF-ID: 1903298).
	

\end{document}